\documentclass[final]{siamltex}
\usepackage{amsfonts,amsmath,amssymb}
\usepackage{mathrsfs,mathtools}
\usepackage{enumerate}
\usepackage{hyperref}
\usepackage{esint}
\usepackage{graphicx}
\usepackage{bm}
\usepackage{esint}
\usepackage{cite}
\usepackage{xcolor}
\usepackage[normalem]{ulem} 
\usepackage{lscape}
\usepackage{subcaption,changepage}

\DeclareMathAlphabet{\mathpzc}{OT1}{pzc}{m}{it}

\usepackage{cleveref}
\usepackage[pass]{geometry}
\usepackage{tikz}
\usepackage{csvsimple}
\usepackage{siunitx}
\usepackage{subcaption}
\usepackage{pgffor}
\usepackage{placeins}

\DeclareMathAlphabet{\mathpzc}{OT1}{pzc}{m}{it}

\newcommand{\Rey}{\mathrm{Re}}

\newcommand{\Gra}{\mathrm{Gr}}

\newcommand{\dd}{\mathrm{d}}

\newcommand{\oo}{\mathrm{o}}

\newcommand{\ii}{\mathrm{i}}
\newcommand{\ww}{\mathrm{w}}
\newcommand{\nn}{\mathrm{n}}

\newcommand{\se}[1]{\times10^{#1}}

\title{On existence and uniqueness of solutions to a Boussinesq system with nonlinear and mixed boundary conditions
}
\author{Rafael Arndt\footnotemark[1]\thanks{Department of Mathematical Sciences and the Center for Mathematics and Artificial Intelligence (CMAI),  George  Mason  University,  Fairfax,  VA  22030,  USA, \texttt{tarndt@gmu.edu} (R. Arndt), \texttt{crautenb@gmu.edu} (C. N. Rautenberg)\vspace{.25cm}}
\and
Andrea N. Ceretani\footnotemark[2]\thanks{Instituto de Investigaciones Matem\'{a}ticas ``Luis A. Santal\'{o}'' and Departamento de Matem\'atica, Facultad de Ciencias Exactas, Universidad de Buenos Aires, Intendente Guiraldes 2160, Buenos Aires 1428, Argentina, \texttt{aceretani@dm.uba.ar}}
\and
Carlos N. Rautenberg\footnotemark[1]
}

\begin{document}  

\maketitle

\begin{abstract}
We study a Boussinesq system in a bounded domain with an outlet boundary
portion where fluid can leave or re-enter. On this boundary part, we consider
a do-nothing condition for the fluid flow, and a new artificial condition
for the heat transfer that couples nonlinearly the fluid velocity and temperature.
The latter can be further adjusted if convective or conductive phenomena
are dominant. We prove existence and, in some cases, uniqueness of weak
solutions to stationary and evolutionary problems by a fixed point strategy
under suitable assumptions on the data. A variety of numerical tests shows
the improved performance of the new artificial condition with respect to
other standard choices in the literature.
\end{abstract}



\section{Introduction}
In this work, we consider a problem for non-isothermal incompressible Newtonian
fluids in a domain with an outlet where the fluid is allowed to leave.
In particular, we consider a recently introduced artificial nonlinear boundary
condition, study the stationary and time evolutionary versions of the problem,
and provide numerical tests that show more accurate performance than other
standard choices. In order to motivate the issue, consider the problem
to describe the temperature, velocity, and pressure of the air in a room
with a heating device on some of its walls and a door or window which is
always open. This may be a subtask of a larger problem where we aim to
achieve both comfort and efficiency of a heating system. Fig.\ref{Room} shows a possible configuration for the above situation and illustrates
the effects of buoyancy on the fluid velocity due to temperature changes.
\begin{figure}[h]
\centering
		\includegraphics[scale=0.3]{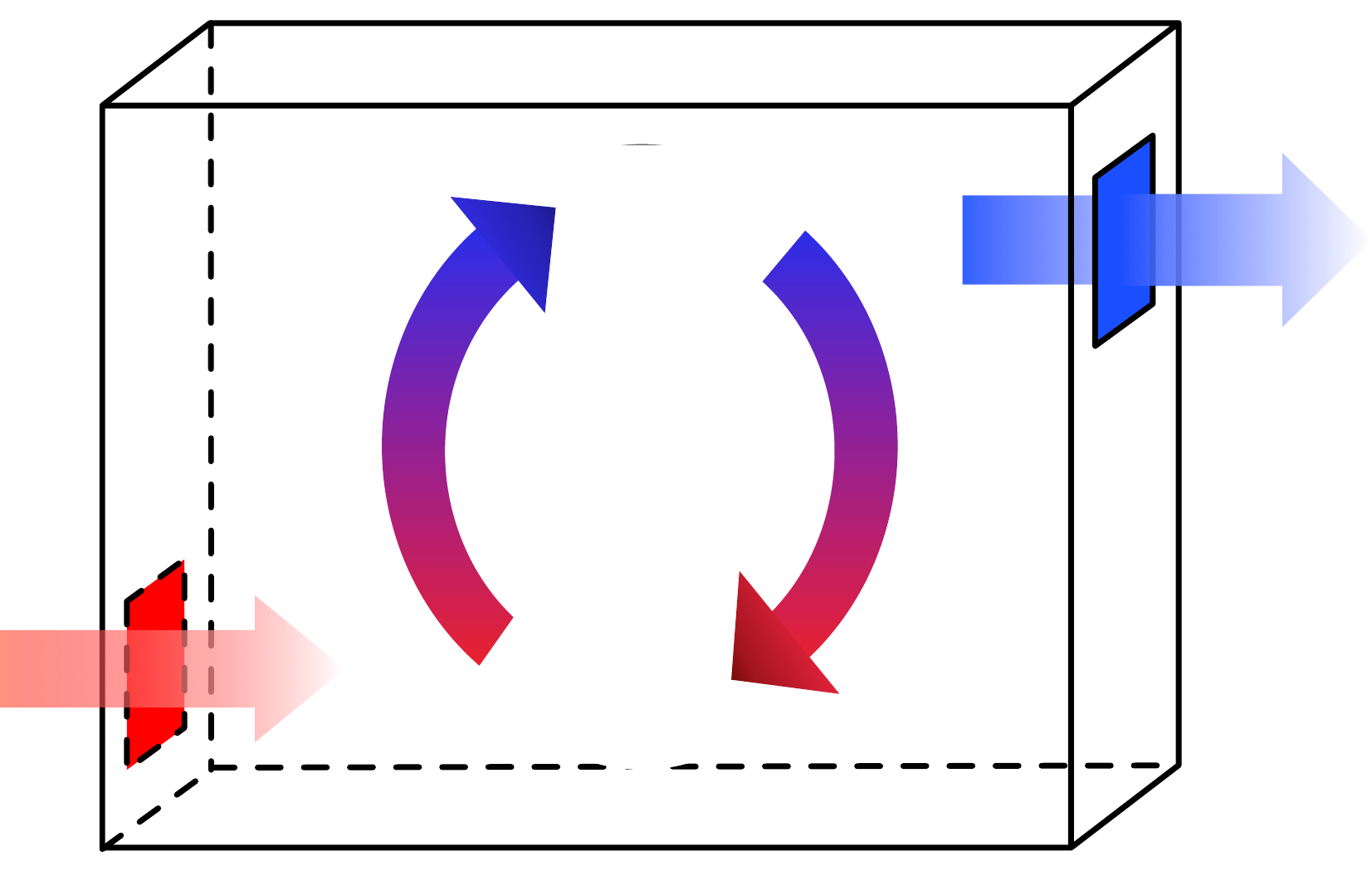}
\caption{\small\small Recirculating flow in a room with a heated inlet flow on the left wall and an aperture on the opposite one.}\label{Room}
\end{figure}
A common approach to treat this problem numerically is to restrict the
computational domain to the room itself and to include the effects of the
surrounding environment through an appropriate selection of boundary conditions.
While this is quite simple to do on the parts of the boundary that represent
insulated walls or a heating device, it is a major problem on the open
door. Particularly, the question that arises is: What conditions on the
velocity and the temperature of the fluid would be appropriate at the place on the boundary
where the fluid is allowed to flow freely? A good perspective of such a
question is given in the words of Gresho and Sani \cite{SaGr1994} when
referring to boundary conditions:
\vspace{.1cm}
\begin{adjustwidth}{2.5cm}{2.5cm}
\begin{center}``\emph{Nature is usually silent, or in fact perverse, in not communicating the appropriate ones.}'' 	
\end{center}
\end{adjustwidth}
\vspace{.1cm}

The problem of selecting artificial boundary conditions for isothermal
incompressible Newtonian fluids took impulse in the early 1990s from the
work \cite{Gr1991} of Gresho (see also \cite{HeRaTu1996,KuSa1998}), and
its high level of difficulty was noticed from the beginning. In
\cite{SaGr1994}, Gresho and Sani collected a set of benchmark problems
treated with different artificial boundary conditions with the aim of finding
appropriate ones. In words of the authors, this goal was not met, and since
then many other works have been made guided by the same target (see, e.g.,
the surveys in \cite{DoKaCh2014} or \cite{KrNe2018}). However, the artificial
boundary condition proposed by Gresho in his pioneering article
\cite{Gr1991} became quite popular. It is known as the ``do-nothing condition''
and, in dimensionless form, reads as follows:
%
\begin{equation}
\label{DN}
-p\,{\mathbf{n}}+\frac{1}{\mathrm{Re}}\,
\frac{\partial {\mathbf{v}}}{\partial {\mathbf{n}}}={\mathbf{h}}\qquad \qquad
\text{on}\quad \Gamma _{\mathrm{o}},
\end{equation}
where $p$ and ${\mathbf{v}}$ are the pressure and the velocity of the fluid,
respectively, $\mathrm{Re}$ is the Reynolds number, ${\mathbf{n}}$ is the unit
outer normal to the open boundary $\Gamma _{\mathrm{o}}$, and
${\mathbf{h}}$ is a given function. The success of the do-nothing condition  {\eqref{DN}} is probably due to its relative easy implementation for numerical
methods based on variational formulations. Furthermore, even if the origin
of the do-nothing condition relies on variational principles (see, e.g.,
\cite{KrNe2018}), it is verified by Poiseuille flows in circular pipes,
and thus it is physically plausible. The do-nothing condition  {\eqref{DN}} has been widely used in the last decade. Successful and insightful
approaches can be found in the works of Bene\v{s}, Ku\v{c}era and collaborators
\cite{Be2007,BeKu2007,Be2009,Ku2009,Be2011-a,BeKu2012-a,Be2011-b,BeKu2012-b,Be2013-a,Be2013-b,Be2014,BeTi2015,BeKu2016,BePa2017}.
Nevertheless, the artificial condition  {\eqref{DN}} has two well-known drawbacks:
One of them is that it cannot prevent an unbounded growth of the kinetic
energy of the system since it allows the fluid to leave and re-enter the
truncated domain at unrestricted rates. As a consequence, a priori energy
estimates (and hence, existence of weak solutions) can be obtained only
for ``small'' data. The other one concerns its numerical instability. A
succinct discussion on both subjects can be found in \cite{BrMu2014} (see
also \cite{BeEtAl2018}). We finally mention that some ad-hoc changes on
the model proposed by Gresho have proven suitable to overcome some drawbacks
of the do-nothing condition. Successful approaches have been obtained by
adding an extra term to the left-hand-side of  {\eqref{DN}} (see
\cite{BrFa1996,FeNe2004,FeNe2006,PeThBlCr2008-a,PeThBlCr2008-b,FeNe2013,BrMu2014,Ne2015,Ne2016,Ne2017,CeRa2019})
as well as by adding an extra condition on the open boundary (see
\cite{KrNe1994,KrNe2001,KrNe2018}). Other artificial boundary conditions
for isothermal fluids can be found in \cite{DoKaCh2014} and the references
therein.

The analogous problem concerning heat-conducting fluids has received less
attention. Some authors have used a do-nothing condition for the velocity
field together with a Neumann condition for the fluid temperature, see
for example
\cite{Be2011-b,BeKu2012-b,Be2013-b,Be2014,BeTi2015,BePa2017,PeThBlCr2008-a,PeThBlCr2008-b}.
In other cases, the two ad-hoc changes on Gresho's model mentioned
before, together with an artificial boundary condition that couples the
temperature and the velocity of the fluid were used; see
\cite{PeThBlCr2008-a,PeThBlCr2008-b,Ne2017,CeRa2019}. In
\cite{CeRa2019}, the authors introduced the condition
%
\begin{equation}
\label{HT}
\frac{1}{\mathrm{Re}\,\mathrm{Pr}}\,
\frac{\partial u}{\partial {\mathbf{n}}}- u\,\beta ({\mathbf{v}}\cdot {\mathbf{n}})({
\mathbf{v}}\cdot {\mathbf{n}})=h\qquad \text{on}\quad \Gamma _{\mathrm{o}},
\end{equation}
where $h$ and $\beta $ are given, $u$ is the temperature of the fluid,
and $\mathrm{Pr}$ is the Prandtl number. The function
$\beta :\mathbb{R}\to \mathbb{R}$ was originally conceived bounded, continuous
everywhere except maybe at the origin, and such that
$\beta (s)\in [0,1/2]$ if $s\geq 0$, and $\beta (s)\in [1/2,1]$ otherwise. The boundary condition  {\eqref{HT}} arises consequently from the study of heat
conduction at the outlet and through an appropriate selection of functions
$h$ and $\beta $: For example, the homogeneous Neumann condition ($h
\equiv 0$ and $\beta \equiv 0$) may be suitable if \emph{advection} dominates
the heat transfer process at the open boundary. However, if
\emph{conduction} is not negligible, a complex heat transfer process takes
place at the open boundary; for example, the outside temperature can not
be longer considered impervious to the domain effects, and might be transported
into the domain by incoming flows. These two mutually dependent processes
contribute to the heat transfer at the outlet and the function
$\beta \not\equiv 0$ aims to capture this feature. The bounds of
$s\mapsto \beta (s)$, for $s<0$ and $s>0$, directly affect the influence
of incoming and outgoing flows, respectively, on the heat transfer at the
open boundary; see \cite{CeRa2019}. We refer to the numerical and experimental
works on open heated cavities reported by Chan and Tien in
\cite{ChTi1985,ChTi1986} for insights on complex phenomena taking place
at the open boundary. The coupled condition used in
\cite{PeThBlCr2008-a,PeThBlCr2008-b,Ne2017} is a special case of  {\eqref{HT}} corresponding to a function $\beta $ given by
$\beta (s)=c/2$ if $s<0$ and $\beta (s)=0$ otherwise, where $c>0$ is
constant. To the best of our knowledge, there are no other approaches for
dealing with artificial boundary conditions for non-isothermal fluids in
open domains.

In this work we study stationary and evolutionary Boussinesq equations
with mixed boundary conditions in a domain with an open boundary. The physical
problem in focus throughout the article is related to the example given
above: Describing the velocity field, the temperature, and the pressure
distributions of a fluid inside a room with a heating device located on
its walls (e.g., a radiator and/or a vent of a heating system) and an open
window or door. At the open boundary, we assume that the velocity, the
temperature, and the pressure of the fluid satisfy the do-nothing condition  {\eqref{DN}} and the coupled condition  {\eqref{HT}}. In particular, we remove the conditions
on $\beta $ associated with the sign of its argument, so that Neumann conditions
are included in our study. For simplicity, we consider
${\mathbf{h}}\equiv 0$ and $h\equiv 0$ in  {\eqref{DN}} and  {\eqref{HT}}, respectively.
Existence of solutions for the stationary problem is obtained via two approaches
and under a ``small'' data assumption. In particular, uniqueness of solutions
is attained if $\beta $ is considered to be Lipschitz continuous, while
existence of solutions is retained for $\beta $ continuous, bounded, and
possibly discontinuous at the origin. The time evolutionary problem is
significantly more complex. In fact, existence and uniqueness of solutions
can be obtained as well for two-dimensional problems provided
$\beta $ is Lipschitz continuous. While on the surface, the approach seems
analogous to the stationary case, the entire mathematical machinery only
works by the careful selection of possible state spaces. The latter efforts
lead additionally to an improved regularity result.

The outline of the article is as follows. In section~\ref{Stationary} we
consider the stationary problem, the weak formulation and the existence
results are provided in sections~\ref{sec:weakformuexis} and \ref{Stationary-WeakSolutions}, respectively. The latter contains two completely
different approaches for existence, while in section~\ref{sec:beta-Lipschitz} we consider a Lipschitz continuous function
$\beta $, in section~\ref{sec:beta-A1} existence is given for
$\beta $ bounded and continuous with the possible exception of the origin.
The evolutionary problem is considered from section~\ref{sec:evo} on, its
weak formulation is provided in \ref{sec:weakformu-evo}, and the existence
and uniqueness result is given in \ref{sec:weakformuexis-evo}. The paper
finalizes with a series of numerical tests in section~\ref{Section:numerics} showing the increased accuracy of the boundary condition
 {(\ref{HT})} with respect to other standard choices.

\section{The stationary problem}\label{Stationary}
Let $\Omega \subset \mathbb{R}^{d}$ be a simply connected bounded domain
with Lipschitz boundary $\Gamma :=\partial \Omega $, where $d=2$ or
$d=3$. In addition, let
$\Gamma _{\mathrm{i}}, \Gamma _{\mathrm{w}}, \Gamma _{\mathrm{d}},
\Gamma _{\mathrm{n}}, \Gamma _{\mathrm{o}}$ be relative open subsets of
$\Gamma $ with positive $(d-1)$-Lebesgue measure that satisfy
\begin{align*}
&\Gamma =\overline{\Gamma _{\mathrm{i}}}\cup
\overline{\Gamma _{\mathrm{w}}} \cup
\overline{\Gamma _{\mathrm{o}}}, & \Gamma _{\mathrm{i}}\cap \Gamma _{\mathrm{w}}=\emptyset , & & \Gamma _{\mathrm{i}}\cap \Gamma _{\mathrm{o}}=\emptyset , & & \Gamma _{\mathrm{o}}\cap \Gamma _{\mathrm{w}}=\emptyset ,
\\
&\Gamma =\overline{\Gamma _{\mathrm{d}}}\cup
\overline{\Gamma _{\mathrm{n}}} \cup
\overline{\Gamma _{\mathrm{o}}},& \Gamma _{\mathrm{d}}
\cap \Gamma _{\mathrm{n}}=\emptyset ,& & \Gamma _{\mathrm{d}}\cap
\Gamma _{\mathrm{o}}=\emptyset ,& & \Gamma _{\mathrm{n}}\cap \Gamma _{\mathrm{o}}=\emptyset .
\end{align*}
For the reference problem in which $\overline{\Omega }$ represents a room
with a heating device on some of its walls and an open window or door,
the decompositions
$\left \{  \Gamma _{\mathrm{i}}, \Gamma _{\mathrm{w}}, \Gamma _{\mathrm{o}}\right \}  $ and
$\left \{  \Gamma _{\mathrm{d}}, \Gamma _{\mathrm{n}}, \Gamma _{\mathrm{o}}\right \}  $ are related to the behavior on the boundary of the
velocity and the temperature of the fluid, respectively; see  {Fig.~\ref{BoundarySubparts}}. The boundary partitions represent the following
phenomena and elements:
\begin{itemize}
\item[--] $\Gamma _{\mathrm{o}}$. An open boundary, e.g., an open window
or door, where fluid is allowed to leave the domain.
\item[--] $\Gamma _{\mathrm{i}}$. The inlet location where fluid is allowed
to enter the domain, e.g., a vent of a heating/cooling system.
\item[--] $\Gamma _{\mathrm{w}}$. Parts of the boundary where fluid does
not enter nor leave the domain and further does not slip, e.g., walls or
floors.
\item[--] $\Gamma _{\mathrm{d}}$. Heated regions on walls or floor, e.g.,
a radiating floor, or inlets where the temperature for the fluid is prescribed.
\item[--] $\Gamma _{\mathrm{n}}$. Thermally insulated parts of the boundary.
\end{itemize}

Then, for example, we may have
$\Gamma _{\mathrm{i}}\subset \Gamma _{\mathrm{d}}$ and
$\Gamma _{\mathrm{n}}\equiv \Gamma _{\mathrm{w}}$. According to this reference
problem, the subscripts ``$\mathrm{i}$'', ``$\mathrm{w}$'', and
``$\mathrm{o}$'' refer to the \emph{inlet}, the \emph{walls}, and the
\emph{open} part of the room, and the subscripts ``$\mathrm{d}$'' and
``$\mathrm{n}$'' refer to the parts of the boundary on which
\emph{Dirichlet} and \emph{Neumann} type conditions shall be respectively
imposed. Furthermore, from now on and throughout the paper,
${\mathbf{v}}$, $u$, and $p$ denote the velocity, the temperature, and the pressure
of the fluid, respectively.
\begin{figure}[h!]
\centering
	\begin{subfigure}{0.42\textwidth}
		\includegraphics[width=\textwidth]{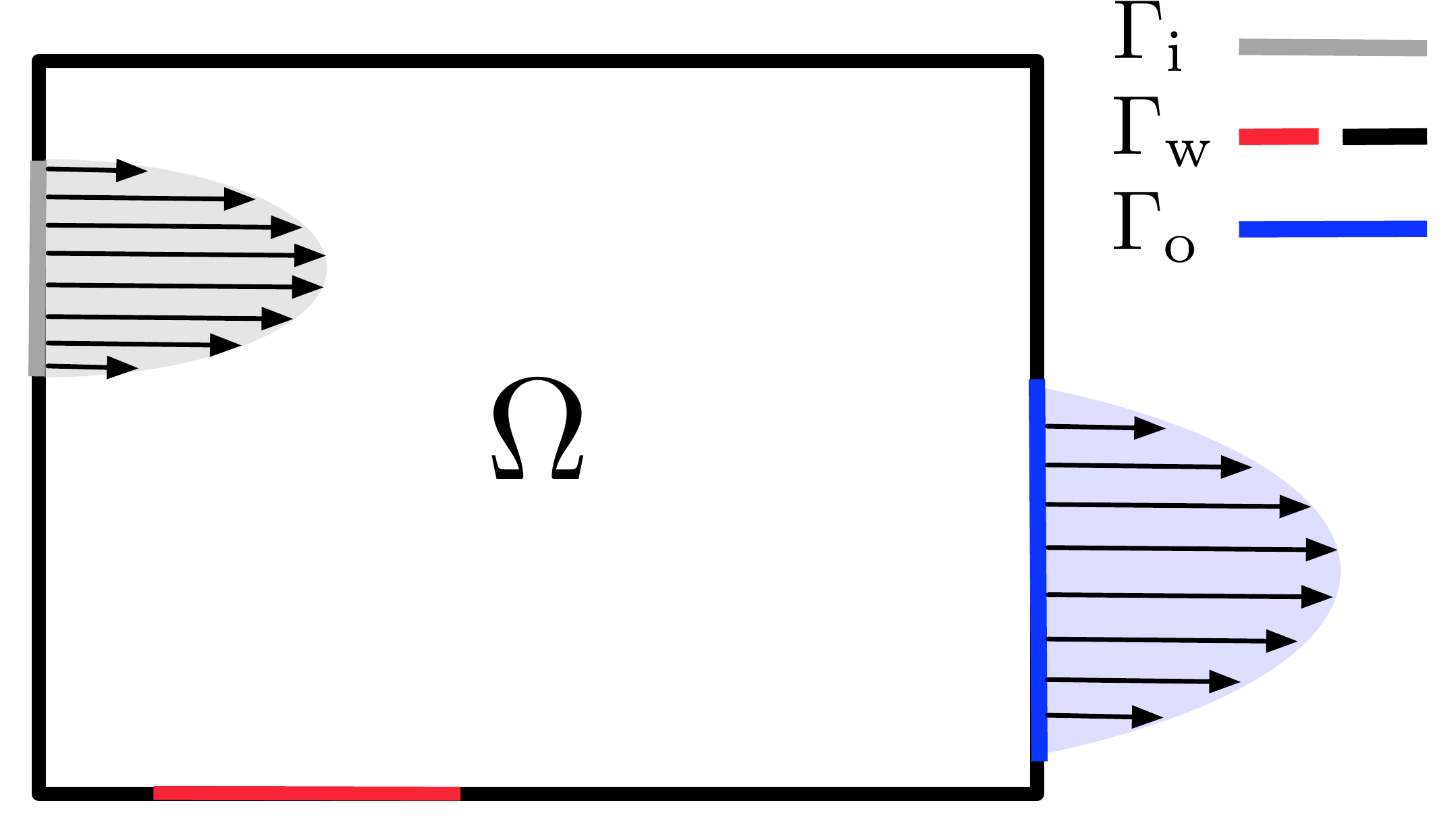}
		\caption{\small Boundary subparts $\Gamma_\ii$, $\Gamma_\ww$, $\Gamma_\oo$.}
	\end{subfigure}\hspace*{1.5	cm}
		\begin{subfigure}{0.44\textwidth}
		\includegraphics[width=\textwidth]{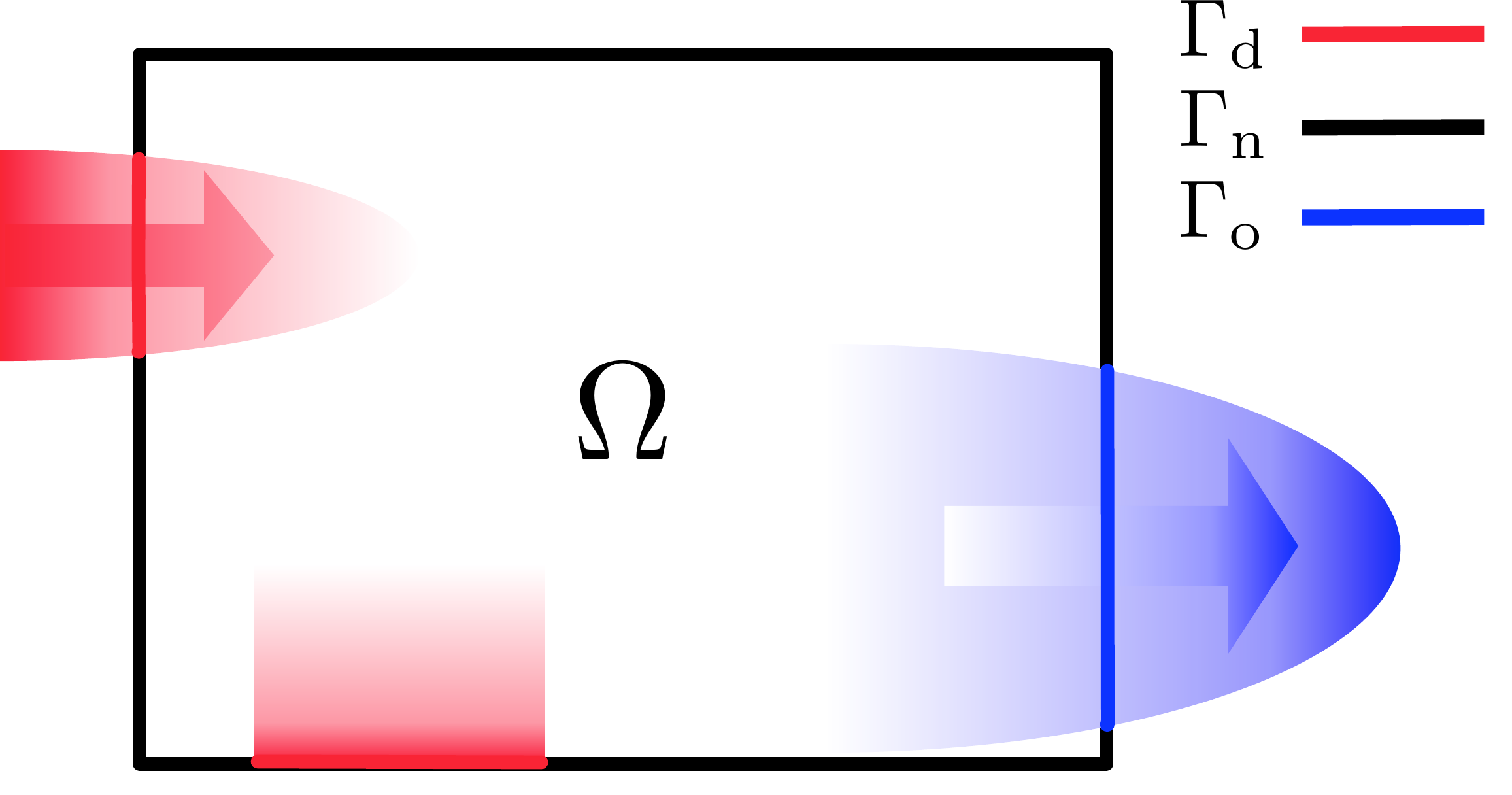}
		\caption{\small Boundary subparts $\Gamma_\dd$, $\Gamma_\nn$, $\Gamma_\oo$.}
	\end{subfigure}
	\caption{\small Possible configuration for 2d-decompositions $\{\Gamma_\ii,\Gamma_\dd,\Gamma_\oo\}$ and $\{\Gamma_\dd,\Gamma_\nn,\Gamma_\oo\}$.}\label{BoundarySubparts}
\end{figure}

The stationary problem of interest is the following.

\bigskip\noindent
\textbf{Problem $(\mathbb{P})$}: \emph{Find
${\mathbf{v}}:\Omega \to \mathbb{R}^{d}$, $u:\Omega \to \mathbb{R}$, and
$p:\Omega \to \mathbb{R}$ that satisfy the stationary Boussinesq equations
in $\Omega $},
%
\begin{align}%
\label{BoussinesqStationary-1}
{\mathbf{v}}\cdot \nabla {\mathbf{v}}-\frac{1}{\mathrm{Re}}\Delta {\mathbf{v}}+
\nabla p=\,&\frac{\mathrm{Gr}}{\mathrm{Re}^{2}}u\,{\mathbf{e}}+{\mathbf{g}}_{1},
\\
\mathrm{div}\,{{\mathbf{v}}}=\,&0,
\\
\label{BoussinesqStationary-3}
{\mathbf{v}}\cdot \nabla u-\frac{1}{\mathrm{Re}\,\mathrm{Pr}}\Delta u=\,&g_{2},
\end{align}
\emph{subject to the boundary conditions}
%
\begin{align}
\label{BoundaryConditionsStationary}
{\mathbf{v}}&={\mathbf{v}}_{\mathrm{i}}\quad \,\text{\emph{on}}\quad \Gamma _{\mathrm{i}},&\qquad {
\mathbf{v}}&=0\quad \,\,\,\text{\emph{on}}\quad \Gamma _{\mathrm{w}},
\\
\frac{\partial u}{\partial {\mathbf{n}}}&=0\quad \,\,\,\,\text{\emph{on}}\quad
\Gamma _{\mathrm{n}},&\qquad \,\, u&=u_{\mathrm{d}}\quad \text{\emph{on}}\quad
\Gamma _{\mathrm{d}},\qquad
\\
\frac{1}{\mathrm{Re}}\,\frac{\partial {\mathbf{v}}}{\partial {\mathbf{n}}}&=p\,{
\mathbf{n}}\quad \text{\emph{on}}\quad \Gamma _{\mathrm{o}},&\qquad
\frac{1}{\mathrm{Re}\,\mathrm{Pr}}\,
\frac{\partial u}{\partial {\mathbf{n}}}&= u\,\beta ({\mathbf{v}}\cdot {\mathbf{n}})\,({
\mathbf{v}}\cdot {\mathbf{n}})\quad \text{\emph{on}}\quad \Gamma _{\mathrm{o}}.
\end{align}
Here, ${\mathbf{g}}_{1}:\Omega \to \mathbb{R}^{d}$ and
$g_{2}:\Omega \to \mathbb{R}$ represent an external force/perturbation
and the intensity of external sources of heat, respectively. The non-negative
constants $\mathrm{Re}$, $\mathrm{Pr}$, $\mathrm{Gr}$ are the Reynolds,
Prandtl, and Grashof numbers. The vector ${\mathbf{e}}$ represents the opposite
direction of the gravitational acceleration, i.e., ${\mathbf{e}}:=(0,1)$ if
$d=2$, or ${\mathbf{e}}:=(0,0,1)$ if $d=3$. The function
${\mathbf{v}}_{\mathrm{i}}:\Gamma _{\mathrm{i}}\to \mathbb{R}^{d}$ is the velocity
profile at the inlet, and
$u_{\mathrm{d}}:\Gamma _{\mathrm{d}}\to \mathbb{R}$, the temperature distribution
at the inlet and/or the heated part of the boundary. The vector
${\mathbf{n}}$ is the outer unit normal to the boundary $\Gamma $, and
$\beta :\mathbb{R}\to \mathbb{R}$ is a given function that satisfies that
\begin{enumerate}
\item[(A1)] $\beta $ is bounded and continuous everywhere except maybe
at the origin.
\end{enumerate}
In just a few words, $\beta $ is a function associated to the thermal conduction
phenomena at the outlet. In particular, for slow moving fluids,
$\beta $ allows to account for heat contributions outside the domain. We
refer to \cite{Gr1991} and \cite{CeRa2019} for details about the boundary
conditions on $\Gamma _{\mathrm{o}}$.

\subsection{Weak formulation}\label{sec:weakformuexis}
In this section we establish the weak formulation of the stationary problem
together with the function spaces of interest and generalities about involved
operators. Let
\begin{align*}
&E_{1}(\Omega ):=\left \{  {\mathbf{y}}\in C^{\infty }(\overline{\Omega })^{d}
\,:\, \mathrm{div}\,{\mathbf{y}}=0,\, \overline{\mathrm{supp}\,{\mathbf{y}}}
\cap (\Gamma _{\mathrm{i}}\cup \Gamma _{\mathrm{w}})=\emptyset \right
\}  ,
\\
&E_{2}(\Omega ):=\left \{  y\in C^{\infty }(\overline{\Omega })\,:\,
\overline{\mathrm{supp}\,y}\cap \Gamma _{\mathrm{d}}=\emptyset
\right \}  ,
\end{align*}
where $C^{\infty }(\overline{\Omega })$ is the set of restrictions to
$\overline{\Omega }$ of the real-valued infinitely differentiable functions
in $\mathbb{R}^{d}$. We define the spaces
\begin{equation*}
\begin{split} V_{1}:&\quad
\text{closure of $E_{1}(\Omega )$ with respect to the norm }\|\cdot \|_{W^{1,2}(
\Omega )^{d}},
\\
V_{2}:&\quad
\text{closure of $E_{2}(\Omega )$ with respect to the norm }\|\cdot \|_{W^{1,2}(
\Omega )},
\\
H_{1}:&\quad
\text{closure of $E_{1}(\Omega )$ with respect to the norm }\|\cdot \|_{L^{2}(
\Omega )^{d}},
\\
H_{2}:&\quad
\text{closure of $E_{2}(\Omega )$ with respect to the norm }\|\cdot \|_{L^{2}(
\Omega )}.
\end{split}
\end{equation*}
Notice that $\|\nabla \cdot \|_{L^{2}(\Omega )^{d\times d}}$ is equivalent
to $\|\cdot \|_{W^{1,2}(\Omega )^{d}}$ in $V_{1}$ and, analogously,
$\|\nabla \cdot \|_{L^{2}(\Omega )^{d}}$ is equivalent to
$\|\cdot \|_{W^{1,2}(\Omega )}$ in $V_{2}$. From now on, we consider
$V_{1}$ and $V_{2}$ endowed with the norms
\begin{equation*}
\|\cdot \|_{V_{1}}:=\|\nabla \cdot \|_{L^{2}(\Omega )^{d\times d}},
\qquad \|\cdot \|_{V_{2}}:=\|\nabla \cdot \|_{L^{2}(\Omega )^{d}}.
\end{equation*}
In addition, we identify the Hilbert spaces $H_{1}$ and $H_{2}$ with their
topological duals $H_{1}'$ and $H_{2}'$ by the Riesz map. Observe that
the embedding $V_{i}\hookrightarrow H_{i}$ is dense and continuous, and
hence $H_{i}\equiv H_{i}'\hookrightarrow V'_{i}$ for $i=1,2$ is as
well; see \cite{GaPa2006}.

We assume that
\begin{enumerate}[(A2)]
\label{cond:A2}
\item[$\mathrm{(A2)}$] $\Omega $ belongs to the domain class
$\tilde{C}\,^{0,1}$ and the boundary parts $\Gamma _{\mathrm{d}}$ and
$\Gamma _{\mathrm{i}}\cup \Gamma _{\mathrm{w}}$ consist of a finite number
of relative open sets in $\Gamma $, which satisfy a projective boundary
Lipschitz regularity condition if $d=3$ (see \cite{DoPa2006}),
\end{enumerate}
where $\tilde{C}\,^{0,1}$ is the subset of the Lipschitz domain class
$C\,^{0,1}$ formed by domains with a Lipschitz boundary consisting of a
finite number of smooth parts with a finite number of relative maxima,
minima, and inflexion points, and in three dimensions, also a finite number
of saddle points (see \cite{DoPa2006}). Hence, the spaces $V_{1}$ and
$V_{2}$ admit the characterizations
\begin{equation*}
\begin{split} &V_{1}=\left \{  {\mathbf{y}}\in W^{1,2}(\Omega )^{d}\,:\,
\mathrm{div}\,{\mathbf{y}}=0\text{ in }\Omega ,\, \left .{\mathbf{y}}\right |_{
\Gamma _{\mathrm{i}}\cup \Gamma _{\mathrm{w}}}=0
\text{ in the trace sense}\right \}  ,
\\
&V_{2}=\left \{  y\in W^{1,2}(\Omega )\,:\, \left .y\right |_{\Gamma _{\mathrm{d}}}=0\text{ in the trace sense}\right \}  .
\end{split}
\end{equation*}
The equivalence between the two definitions of $V_{2}$ was proved in
\cite{DoPa2006}. The analogous result for $V_{1}$ is obtained similarly.

We further assume that
\begin{enumerate}
\item[$\mathrm{(A3)}$]
${\mathbf{v}}_{\mathrm{i}}\in W^{1/2,2}(\Gamma _{\mathrm{i}})^{d}$ \text{and} $u_{\mathrm{d}}\in W^{1/2,2}(\Gamma _{\mathrm{d}})$.
\end{enumerate}

This allows us to look for weak solutions to problem $(\mathbb{P})$ of
the form
\begin{equation*}
{\mathbf{v}}=\bar {\mathbf{v}}+{\mathbf{V}},\qquad u=\bar {u}+U,
\end{equation*}
with $\bar {\mathbf{v}}\in V_{1}$ and $\bar {u}\in V_{2}$, where
${\mathbf{V}}\in W^{1,2}(\Omega )^{d}$ and $U\in W^{1,2}(\Omega )$ satisfy
\begin{equation*}
\begin{split} {\mathbf{V}}=\,&{\mathbf{v}}_{\mathrm{i}}\quad \text{on}\quad
\Gamma _{\mathrm{i}},\qquad {\mathbf{V}}=0\quad \text{on}\quad \Gamma _{\mathrm{w}},\qquad \mathrm{div}\,{\mathbf{V}}=0\quad \text{in}\quad
\Omega ,\qquad U=\,u_{\mathrm{d}}\quad \text{on }\Gamma _{\mathrm{d}},
\end{split}
\end{equation*}
and the estimates
%
\begin{equation}
\label{UV-Estimates}
\|{\mathbf{V}}\|_{W^{1,2}(\Omega )^{d}}\leq C\,\|{\mathbf{v}}_{\mathrm{i}}\|_{W^{1/2,2}(
\Gamma _{\mathrm{i}})^{d}},\qquad \|U\|_{W^{1,2}(\Omega )}\leq C\,\|u_{\mathrm{d}}\|_{W^{1/2,2}(\Gamma _{\mathrm{d}})}
\end{equation}
(see, e.g., \cite{Ga2011}). Here and in what follows, $C$ denotes a positive
constant taking various non-essential values that may depend on the domain
$\Omega $, the dimension $d$, and the sets involved in the decompositions
$\{\Gamma _{\mathrm{i}},\Gamma _{\mathrm{w}},\Gamma _{\mathrm{o}}\}$ or
$\{\Gamma _{\mathrm{d}},\Gamma _{\mathrm{n}},\Gamma _{\mathrm{o}}\}$.

We shall study a weak formulation of $(\mathbb{P})$ that is obtained from
the usual approach: We formally multiply equations  {\eqref{BoussinesqStationary-1}} and  {\eqref{BoussinesqStationary-3}} by test
functions in $V_{1}$ and $V_{2}$, respectively. Then, we integrate over
$\Omega $, and finally use the boundary conditions on
$\Gamma _{\mathrm{n}}$ and $\Gamma _{\mathrm{o}}$ to rewrite the integrals
with second order derivatives using Green's theorem. In order to introduce
the weak formulation, we introduce the operators below. In particular,
the estimations given for each of them will be useful to prove the existence
of weak solutions in the next section~\ref{Stationary-WeakSolutions}. In
what follows, $(\cdot ,\cdot )_{2}$ denotes the usual inner product in
either $L^{2}(\Omega )$ or in any finite copies of it.

Let ${\mathbf{w}}_{1},{\mathbf{w}}_{2},{\mathbf{w}}\in W^{1,2}(\Omega )^{d}$ and
$w\in W^{1,2}(\Omega )$. We define
$B_{1}({\mathbf{w}}_{1},{\mathbf{w}}_{2})\in V_{1}'$ and
$B_{2}({\mathbf{w}},w)\in V_{2}'$ by
%
\begin{equation}
\label{B}
\langle B_{1}({\mathbf{w}}_{1},{\mathbf{w}}_{2}),{\mathbf{y}}\rangle :=\int _{\Omega }({
\mathbf{w}}_{1}\cdot \nabla {\mathbf{w}}_{2})\cdot {\mathbf{y}}\,\mathrm{d}x\quad
\text{and}\quad \langle B_{2}({\mathbf{w}},w),y\rangle :=\int _{\Omega }({
\mathbf{w}}\cdot \nabla w)y\,\mathrm{d}x,
\end{equation}
for all $\mathbf{y}\in V_1$ and $y\in V_2$.

By H\"{o}lder's inequality and the embedding
$W^{1,2}(\Omega )\hookrightarrow L^{4}(\Omega )$, we find that
%
\begin{align}
\label{B-L4}
|\langle B_{1}({\mathbf{w}}_{1},{\mathbf{w}}_{2}),{\mathbf{y}}\rangle |\leq \,&\|{
\mathbf{w}}_{1}\|_{L^{4}(\Omega )^{d}}\|\nabla {\mathbf{w}}_{2}\|_{L^{2}(
\Omega )^{d\times d}}\|{\mathbf{y}}\|_{L^{4}(\Omega )^{d}},
\\
|\langle B_{2}({\mathbf{w}},w),y\rangle |\leq \,&\|{\mathbf{w}}\|_{L^{4}(
\Omega )^{d}}\|\nabla w\|_{L^{2}(\Omega )^{d}}\|y\|_{L^{4}(\Omega )}.
\end{align}
Hence, we observe the estimates
%
\begin{align}
\label{B1-W12}
\|B_{1}({\mathbf{w}}_{1},{\mathbf{w}}_{2})\|_{V_{1}'}\leq \,&C\,\|{\mathbf{w}}_{1}\|_{W^{1,2}(
\Omega )^{d}}\|{\mathbf{w}}_{2}\|_{W^{1,2}(\Omega )^{d}},
\\
\label{B2-W12}
\|B_{2}({\mathbf{w}},w)\|_{V_{2}'}\leq \,&C\,\|{\mathbf{w}}\|_{W^{1,2}(\Omega )^{d}}
\|w\|_{W^{1,2}(\Omega )}.
\end{align}
In addition, we define $H({\mathbf{w}},w)$ by
%
\begin{equation}
\label{H}
\langle H({\mathbf{w}},w),y\rangle :=\int _{\Gamma _{\mathrm{o}}}w\beta ({
\mathbf{w}}\cdot {\mathbf{n}})({\mathbf{w}}\cdot {\mathbf{n}})y\,\mathrm{d}\sigma
\end{equation}
for $y\in V_{2}$. Notice that H\"{o}lder's inequality and the embedding
$W^{1,2}(\Omega )\hookrightarrow L^{4}(\Gamma _{\mathrm{o}})$ yield
%
\begin{equation}
\label{H-L4}
|\langle H({\mathbf{w}},w),y\rangle |\leq \|\beta \|_{L^{\infty }(\mathbb{R})}
\|{\mathbf{w}}\|_{L^{4}(\Gamma _{\mathrm{o}})^{d}}\|w\|_{L^{4}(\Gamma _{\mathrm{o}})}\|y\|_{L^{2}(\Gamma _{\mathrm{o}})}.
\end{equation}
Therefore, $H({\mathbf{w}},w)\in V_{2}'$ and satisfies
%
\begin{equation}
\label{H-W12}
\|H({\mathbf{w}},w)\|_{V_{2}'}\leq \,C\,\|\beta \|_{L^{\infty }(\mathbb{R})}
\|{\mathbf{w}}\|_{W^{1,2}(\Omega )^{d}}\|w\|_{W^{1,2}(\Omega )}.
\end{equation}
Finally, we define $A_{1}({\mathbf{w}},w)\in V_{1}'$ and
$A_{2}(w)\in V_{2}'$ by
%
\begin{align}
\label{A}
\langle A_{1}({\mathbf{w}},w),{\mathbf{y}}\rangle :=\,&\frac{1}{\mathrm{Re}}(
\nabla {\mathbf{w}},\nabla {\mathbf{y}})_{2}-
\frac{\mathrm{Gr}}{\mathrm{Re}^{2}}(w{\mathbf{e}},{\mathbf{y}})_{2},
\\
\langle A_{2}(w),y\rangle :=\,&\frac{1}{\mathrm{Re}\,\mathrm{Pr}}(
\nabla w,\nabla y)_{2},
\end{align}
and observe the estimates
%
\begin{align}
\label{A1}
\|A_{1}({\mathbf{w}},w)\|_{V_{1}'}\leq \,&C\left (\frac{1}{\mathrm{Re}}\|{
\mathbf{w}}\|_{W^{1,2}(\Omega )^{d}}+\frac{\mathrm{Gr}}{\mathrm{Re}^{2}}
\|w\|_{W^{1,2}(\Omega )}\right ),
\\
\label{A2}
\|A_{2}(w)\|_{V_{2}'}\leq \,&\frac{C}{\mathrm{Re}\,\mathrm{Pr}}\|w
\|_{W^{1,2}(\Omega )}.
\end{align}

We are now in a position to write the weak formulation of
$(\mathbb{P})$:

\bigskip\noindent
\textbf{Problem $(\mathbb{P}_{\mathrm{w}})$}: \emph{Find
$(\bar {\mathbf{v}},\bar {u})\in V_{1}\times V_{2}$ such that
\begin{align*}
A_{1}(\bar {\mathbf{v}}+{\mathbf{V}},\bar {u}+U)+B_{1}(\bar {\mathbf{v}}+{\mathbf{V}},
\bar {\mathbf{v}}+{\mathbf{V}})=&\,{\mathbf{g}}_{1}&\text{in}&\quad V_{1}',
\\
A_{2}(\bar {u}+U)+B_{2}(\bar {\mathbf{v}}+{\mathbf{V}},\bar {u}+U)-H(
\bar {\mathbf{v}}+{\mathbf{V}},\bar {u}+U)=&\,g_{2}&\text{in}&\quad V_{2}',
\end{align*}
where ${\mathbf{g}}_{1}\in V_{1}'$, $g_{2}\in V_{2}'$.}

\bigskip
A weak solution to $(\mathbb{P})$ is then defined as
$(\bar {\mathbf{v}}+{\mathbf{V}},\bar {u}+U)$ where $(\bar {\mathbf{v}},\bar {u})$ solves
$(\mathbb{P}_{\mathrm{w}})$.

\subsection{Existence of weak solutions}
\label{Stationary-WeakSolutions}

We provide two different proofs for the existence of solutions to
$(\mathbb{P}_{\mathrm{w}})$ associated to different assumptions on
$\beta $. Initially, we consider the case that $\beta $ is Lipschitz continuous
in  {Theorem~\ref{Stationary-Theorem1}}, where we prove also uniqueness of
the solution. The argument of the proof relies on a fixed point strategy
that uses Banach's fixed point theorem. Then, in the case when
$\beta $ is continuous (but not Lipschitz), we prove existence in  {Theorem~\ref{Stationary-Theorem2}} by a perturbation approach combined with Leray-Schauder's
fixed point theorem. Both theorems are proved under the assumption of ``small''
boundary data, small $\mathrm{Re}$ and $\mathrm{Gr}$ numbers, and (possibly)
large $\mathrm{Pr}$ in the sense given in  {Theorem~\ref{Stationary-Theorem1}}.

In the following, we consider the product space $V_{1}\times V_{2}$ endowed
with the norm
\begin{equation*}
\|({\mathbf{w}},w)\|_{V_{1}\times V_{2}}:=\|{\mathbf{w}}\|_{V_{1}}+\|w\|_{V_{2}},
\end{equation*}
and denote by $B_{\tau }$ the closed ball in $V_{1}\times V_{2}$ with center
at the origin and radius $\tau >0$.

\subsubsection{Existence and uniqueness for $\beta $ Lipschitz continuous}
\label{sec:beta-Lipschitz}

In this section, we provide an existence and uniqueness result built on
the assumption of smallness of the data $\mathbf{v}_{\mathrm{i}}$ and
$u_{\mathrm{d}}$, and of enough regularity of $\beta $. The proof is based
on re-writing the weak formulation as a fixed point equation and is given
next.

\begin{theorem}%
\label{Stationary-Theorem1}
Let $\beta :\mathbb{R}\to \mathbb{R}$ be Lipschitz continuous. Then, provided
that
\begin{equation*}
\|{\mathbf{v}}_{i}\|_{W^{1/2,2}(\Gamma _{\mathrm{i}})^{d}}\qquad\text{and}\qquad \|u_{\mathrm{d}}\|_{W^{1/2,2}(\Gamma _{\mathrm{d}})}
\end{equation*}
are sufficiently small, there exist
$\varepsilon _{1}, \varepsilon _{2}, \varepsilon _{3}>0$ such that if
%
\begin{equation}
\label{Stationary-RePrGr}
\mathrm{Re}\in (0,\varepsilon _{1}),\qquad \mathrm{Pr}\in \left (0,
\frac{\varepsilon _{2}}{\mathrm{Re}}\right ),\qquad \mathrm{Gr}\in
\left (0,\varepsilon _{3}\,\mathrm{Re}\right ),
\end{equation}
there exists a unique solution $(\bar {\mathbf{v}},\bar {u})$ to problem
$(\mathbb{P}_{\mathrm{w}})$ in some closed ball $B_{\tau }$.
\end{theorem}

\begin{proof}
For convenience, we write the equations in problem
$(\mathbb{P}_{\mathrm{w}})$ as follows:
\begin{align*}
\frac{1}{\mathrm{Re}}(\nabla \bar {\mathbf{v}},\nabla {\mathbf{y}})_{2}=\,&\langle {
\mathbf{g}}_{1},{\mathbf{y}}\rangle - \langle B_{1}(\bar {\mathbf{v}}+{\mathbf{V}},
\bar {\mathbf{v}}+{\mathbf{V}}),{\mathbf{y}}\rangle -\langle A_{1}({\mathbf{V}},\bar {u}+U),{
\mathbf{y}}\rangle &\,\forall \,&{\mathbf{y}}\in V_{1},
\\
\frac{1}{\mathrm{Re}\,\mathrm{Pr}}(\nabla \bar {u},\nabla y)_{2}=\,&\langle g_{2},y
\rangle - \langle B_{2}(\bar {\mathbf{v}}+{\mathbf{V}},\bar {u}+U),y\rangle -
\langle A_{2}(U),y\rangle
\\
&+\langle H(\bar {\mathbf{v}}+{\mathbf{V}},\bar {u}+U),y\rangle &\,\forall \,& y
\in V_{2}.
\end{align*}
From this, we observe that a necessary and sufficient condition for
$(\bar {\mathbf{v}},\bar {u})$ to be a solution to
$(\mathbb{P}_{\mathrm{w}})$ is that
\begin{equation*}
F(\bar {{\mathbf{v}}},\bar {u})=(\bar {{\mathbf{v}}},\bar {u})
\end{equation*}
for the map $F:=(S_{1}^{-1}\,P_{1}, S_{2}^{-1}\,P_{2})$, where
$P_{1}:V_{1}\times V_{2}\to V_{1}'$ and
$P_{2}:V_{1}\times V_{2}\to V_{2}'$ are defined as
%
\begin{align}
\label{P}
P_{1}({\mathbf{w}},w):=\,&{\mathbf{g}}_{1}- B_{1}({\mathbf{w}}+{\mathbf{V}},{\mathbf{w}}+{\mathbf{V}})-A_{1}({
\mathbf{V}},w+U),
\\
P_{2}({\mathbf{w}},w):=\,& g_{2}-B_{2}({\mathbf{w}}+{\mathbf{V}},w+U)-A_{2}(U)+H({
\mathbf{w}}+{\mathbf{V}},w+U),
\end{align}
and $S_{1}:V_{1}\to V_{1}'$, $S_{2}:V_{2}\to V_{2}'$ are given by
%
\begin{equation}
\label{S}
\langle S_{1}({\mathbf{w}}),{\mathbf{y}}\rangle =\frac{1}{\mathrm{Re}}(\nabla {
\mathbf{w}},\nabla {\mathbf{y}})_{2}\qquad \text{and}\qquad \langle S_{2}(w),y
\rangle =\frac{1}{\mathrm{Pr}\,\mathrm{Re}}(\nabla w,\nabla y)_{2}.
\end{equation}
We notice that $V_{1}$ and $V_{2}$ are Hilbert spaces with the inner product
$(\nabla \cdot ,\nabla \cdot )_{2}$ and hence the inverse maps
$S_{1}^{-1}:V_{1}'\to V_{1}$ and $S_{2}^{-1}:V_{2}'\to V_{2}$ exist by
virtue of Riesz's theorem. Therefore, $F$ is well-defined.

The rest of the proof consists of proving that $F$ is a contraction from
some closed ball $B_{\tau }$ into itself. We introduce the following constants
to simplify notation:
\begin{align*}
\ell _{{\mathbf{v}}_{\mathrm{i}}}:=\,&\|{\mathbf{v}}_{\mathrm{i}}\|_{W^{1/2,2}(
\Gamma _{\mathrm{i}})^{d}},&\quad \ell _{u_{\mathrm{d}}}:=\,&\|u_{\mathrm{d}}\|_{W^{1/2,2}(\Gamma _{\mathrm{d}})},&\quad \ell _{\beta }:=
\|\beta \|_{L^{\infty }(\mathbb{R})},
\\
\ell _{{\mathbf{g}}_{1}}:=\,&\|{\mathbf{g}}_{1}\|_{V_{1}'},&\quad \ell _{g_{2}}:=
\,&\|g_{2}\|_{V_{2}'}.
\end{align*}

Let $\tau >0$, and consider
$({\mathbf{w}}_{1},w_{1}),({\mathbf{w}}_{2},w_{2})\in B_{\tau }$. From the linearity
of $S_{1}$ and $S_{2}$, and the identities
\begin{equation*}
\|S_{1}^{-1}({\mathbf{h}}_{1})\|_{V_{1}}=\mathrm{Re}\,\|{\mathbf{h}}_{1}\|_{V_{1}'}
\quad \forall \,{\mathbf{h}}_{1}\in V_{1}',\qquad \|S_{2}^{-1}(h_{2})\|_{V_{2}}=
\mathrm{Pr}\,\mathrm{Re}\,\|{h_{2}}\|_{V_{2}'}\quad \forall \,h_{2}
\in V_{2}',
\end{equation*}
we find that
%
\begin{equation}
\label{eq:Festimate}
\begin{split} &\|F({\mathbf{w}}_{1},w_{1})-F({\mathbf{w}}_{2},w_{2})\|_{V_{1}
\times V_{2}}
\\
&\quad =\mathrm{Re}\,\|P_{1}({\mathbf{w}}_{1},w_{1})-P_{1}({\mathbf{w}}_{2},w_{2})
\|_{V_{1}'}+\mathrm{Pr}\,\mathrm{Re}\,\|P_{2}({\mathbf{w}}_{1},w_{1})-P_{2}({
\mathbf{w}}_{2},w_{2})\|_{V_{2}'}.
\end{split}
\end{equation}
Further, we observe that
\begin{equation*}
\begin{split} &P_{1}({\mathbf{w}}_{1},w_{1})-P_{1}({\mathbf{w}}_{2},w_{2})
\\
&\quad =B_{1}({\mathbf{w}}_{2}+{\mathbf{V}},{\mathbf{w}}_{1}-{\mathbf{w}}_{2}) +B_{1}({
\mathbf{w}}_{1}-{\mathbf{w}}_{2},{\mathbf{w}}_{1}+{\mathbf{V}})-A_{1}(0,w_{1}-w_{2}).
\end{split}
\end{equation*}
Then, using the estimates  {\eqref{B1-W12}} and  {\eqref{A1}} we obtain that
%
\begin{equation}
\label{eq:F1estimate}
\begin{split} &\|P_{1}({\mathbf{w}}_{1},w_{1})-P_{1}({\mathbf{w}}_{2},w_{2})\|_{V_{1}'}
\\
&\quad \leq C(\|{\mathbf{w}}_{2}+{\mathbf{V}}\|_{W^{1,2}(\Omega )^{d}}+\|{\mathbf{w}}_{1}+{
\mathbf{V}}\|_{W^{1,2}(\Omega )^{d}})\|{\mathbf{w}}_{1}-{\mathbf{w}}_{2}\|_{V_{1}}+C
\,\frac{\mathrm{Gr}}{\mathrm{Re}^{2}}\|w_{1}-w_{2}\|_{V_{2}}
\\
&\quad \leq C\left (\tau +\ell _{{\mathbf{v}}_{\mathrm{i}}}+
\frac{\mathrm{Gr}}{\mathrm{Re}^{2}}\right )\|({\mathbf{w}}_{1},w_{1})-({
\mathbf{w}}_{2},w_{2})\|_{V_{1}\times V_{2}}.
\end{split}
\end{equation}
We shall now estimate
$\|P_{2}({\mathbf{w}}_{1},w_{1})-P_{2}({\mathbf{w}}_{2},w_{2})\|_{V_{2}'}$. We write
\begin{equation*}
P_{2}({\mathbf{w}}_{1},w_{1})-P_{2}({\mathbf{w}}_{2},w_{2})=I+J,
\end{equation*}
where
\begin{align*}
I:=\,&B_{2}({\mathbf{w}}_{2}+{\mathbf{V}},w_{1}-w_{2})+B_{2}({\mathbf{w}}_{1}-{\mathbf{w}}_{2},w_{1}+U),
\\
J:=\,&H({\mathbf{w}}_{1}+{\mathbf{V}},w_{1}+U)-H({\mathbf{w}}_{2}+{\mathbf{V}},w_{2}+U).
\end{align*}
Similarly as we estimated
$\|P_{1}({\mathbf{w}}_{1},w_{1})-P_{1}({\mathbf{w}}_{2},w_{2})\|_{V_{1}'}$, we have
\begin{equation*}
\begin{split} &\|I\|_{V_{2}'}\leq C\left (\tau +\ell _{{\mathbf{v}}_{\mathrm{i}}}+\ell _{u_{\mathrm{d}}}\right )\|({\mathbf{w}}_{1},w_{1})-({
\mathbf{w}}_{2},w_{2})\|_{V_{1}\times V_{2}}.
\end{split}
\end{equation*}
Let $y\in V_{2}$. We consider
\begin{equation*}
\langle H({\mathbf{w}}_{1}+{\mathbf{V}},w_{1}+U),y\rangle -\langle H({\mathbf{w}}_{2}+{
\mathbf{V}},w_{2}+U),y\rangle =J_{1}+J_{2},
\end{equation*}
where
\begin{align*}
J_{1}:=\,&\langle H({\mathbf{w}}_{1}+{\mathbf{V}},w_{1}+U),y\rangle -\langle H({
\mathbf{w}}_{1}+{\mathbf{V}},w_{2}+U),y\rangle ,
\\
J_{2}:=\,&\langle H({\mathbf{w}}_{1}+{\mathbf{V}},w_{2}+U),y\rangle -\langle H({
\mathbf{w}}_{2}+{\mathbf{V}},w_{2}+U),y\rangle .
\end{align*}
From H\"{o}lder's inequality and the embeddings
$W^{1,2}(\Omega )\hookrightarrow L^{p}(\Gamma _{\mathrm{o}})$ for
$p=2$ and $p=4$, we obtain that
%
\begin{equation}
\label{Stationary-J1}
\begin{split} |J_{1}|\leq \,&\ell _{\beta }\int _{\Gamma _{\mathrm{o}}}|w_{1}-w_{2}||{
\mathbf{w}}_{1}+{\mathbf{V}}||y|\,\mathrm{d}\sigma
\\
\leq \,&\ell _{\beta }\|w_{1}-w_{2}\|_{L^{4}(\Gamma _{\mathrm{o}})}\|{
\mathbf{w}}_{1}+{\mathbf{V}}\|_{L^{4}(\Gamma _{\mathrm{o}})^{d}}\|y\|_{L^{2}(
\Gamma _{\mathrm{o}})}
\\
\leq \,&C\,\ell _{\beta }\|w_{1}-w_{2}\|_{V_{2}}(\|{\mathbf{w}}_{1}\|_{V_{1}}+
\ell _{{\mathbf{v}}_{\mathrm{i}}})\|y\|_{V_{2}}.
\end{split}
\end{equation}
Analogously, defining
$F_{\beta }({\mathbf{w}}_{i},{\mathbf{w}}_{j}):=\beta (({\mathbf{w}}_{i}+{\mathbf{V}})\cdot {
\mathbf{n}})(({\mathbf{w}}_{j}+{\mathbf{V}})\cdot {\mathbf{n}})$ for $i,j=1,2$, we have
%
\begin{equation}
\label{Stationary-J2}
\begin{split} |J_{2}|\leq \,&\int _{\Gamma _{\mathrm{o}}}|w_{2}+U||F_{\beta }({\mathbf{w}}_{1},{\mathbf{w}}_{1})-F_{\beta }({\mathbf{w}}_{1},{\mathbf{w}}_{2})||y|
\,\mathrm{d}\sigma
\\
&+\int _{\Gamma _{\mathrm{o}}}|w_{2}+U||F_{\beta }({\mathbf{w}}_{1},{\mathbf{w}}_{2})-F_{\beta }({\mathbf{w}}_{2},{\mathbf{w}}_{2})||y|\,\mathrm{d}\sigma
\\
\leq \,&\ell _{\beta }\int _{\Gamma _{\mathrm{o}}}|w_{2}+U||{\mathbf{w}}_{1}-{
\mathbf{w}}_{2}||y|\,\mathrm{d}\sigma +L\int _{\Gamma _{\mathrm{o}}}|w_{2}+U||{
\mathbf{w}}_{2}+{\mathbf{V}}||{\mathbf{w}}_{1}-{\mathbf{w}}_{2}||y|\,\mathrm{d}\sigma
\\
\leq \,&C\ell _{\beta }\,(\|w_{2}\|_{V_{2}}+\ell _{u_{\mathrm{d}}})\|{
\mathbf{w}}_{1}-{\mathbf{w}}_{2}\|_{V_{1}}\|y\|_{V_{2}}
\\
&+CL\,(\|w_{2}\|_{V_{2}}+\ell _{u_{\mathrm{d}}})(\|{\mathbf{w}}_{2}\|_{V_{1}}+
\ell _{{\mathbf{v}}_{\mathrm{i}}})\|{\mathbf{w}}_{1}-{\mathbf{w}}_{2}\|_{V_{1}}\|y\|_{V_{2}},
\end{split}
\end{equation}
where $L>0$ is the Lipschitz constant of $\beta $.

From the bounds on $J_{1}$ and $J_{2}$, it follows that
\begin{equation*}
\begin{split} \|J\|_{V_{2}'}\leq \,&C(L\tau ^{2}+(\ell _{\beta }+L(\ell _{{
\mathbf{v}}_{\mathrm{i}}}+\ell _{u_{\mathrm{d}}}))\tau + \ell _{\beta }(\ell _{{
\mathbf{v}}_{\mathrm{i}}}+\ell _{u_{\mathrm{d}}})+L\ell _{{\mathbf{v}}_{\mathrm{i}}}\ell _{u_{\mathrm{d}}})
\\
&\times \|({\mathbf{w}}_{1},w_{1})-({\mathbf{w}}_{2},w_{2})\|_{V_{1}\times V_{2}},
\end{split}
\end{equation*}
and thus
%
\begin{equation}
\label{eq:F2estimate}
\begin{split} &\|P_{2}({\mathbf{w}}_{1},w_{1})-P_{2}({\mathbf{w}}_{2},w_{2})\|_{V_{2}'}
\\
&\quad \leq C(L\tau ^{2}+(1+\ell _{\beta }+L(\ell _{{\mathbf{v}}_{\mathrm{i}}}+
\ell _{u_{\mathrm{d}}}))\tau + (1+\ell _{\beta })(\ell _{{\mathbf{v}}_{\mathrm{i}}}+\ell _{u_{\mathrm{d}}})+L\ell _{{\mathbf{v}}_{\mathrm{i}}}\ell _{u_{\mathrm{d}}})
\\
&\qquad \times \|({\mathbf{w}}_{1},w_{1})-({\mathbf{w}}_{2},w_{2})\|_{V_{1}
\times V_{2}}.
\end{split}
\end{equation}

Hence, using  {\eqref{eq:F1estimate}} and  {\eqref{eq:F2estimate}} on  {\eqref{eq:Festimate}}, we obtain that
\begin{equation*}
\begin{split} &\|F({\mathbf{w}}_{1},w_{1})-F({\mathbf{w}}_{2},w_{2})\|_{V_{1}
\times V_{2}}\leq K_{1}\|({\mathbf{w}}_{1},w_{1})-({\mathbf{w}}_{2},w_{2})\|_{V_{1}
\times V_{2}},
\end{split}
\end{equation*}
where
\begin{equation*}
\begin{split} K_{1}:=\,&C\,\mathrm{Re}\,\mathrm{Pr}\,L\tau ^{2}+C(
\mathrm{Re}+\mathrm{Re}\,\mathrm{Pr}\,(1+\ell _{\beta }+L(\ell _{{
\mathbf{v}}_{\mathrm{i}}}+\ell _{u_{\mathrm{d}}})))\tau
\\
&+C\left (\mathrm{Re}\,\ell _{{\mathbf{v}}_{\mathrm{i}}}+
\frac{\mathrm{Gr}}{\mathrm{Re}}+\mathrm{Re}\,\mathrm{Pr}\,(1+
\ell _{\beta })(\ell _{{\mathbf{v}}_{\mathrm{i}}}+\ell _{u_{\mathrm{d}}})+
\mathrm{Re}\,\mathrm{Pr}\,L\ell _{{\mathbf{v}}_{\mathrm{i}}}\ell _{u_{\mathrm{d}}}\right ).
\end{split}
\end{equation*}

Analogous computations yield
\begin{equation*}
\begin{split} \|F({\mathbf{w}},w)\|_{V_{1}\times V_{2}} \leq \,& K_{2}
\end{split}
\end{equation*}
for any $({\mathbf{w}},w)\in B_{\tau }$, where
\begin{equation*}
\begin{split} K_{2}:=\,&C(\mathrm{Re}+\mathrm{Re}\,\mathrm{Pr}\,(1+
\ell _{\beta }))\tau ^{2}+C\left (\frac{\mathrm{Gr}}{\mathrm{Re}}+
\mathrm{Re}\,\mathrm{Pr}\,(1+\ell _{\beta })(\ell _{{\mathbf{v}}_{\mathrm{i}}}+\ell _{u_{\mathrm{d}}})\right )\tau
\\
&+C\left (\mathrm{Re}\,(\ell _{{\mathbf{g}}_{1}}+\ell _{{\mathbf{v}}_{\mathrm{i}}}^{2})+\frac{\mathrm{Gr}}{\mathrm{Re}}\ell _{u_{\mathrm{d}}}+\mathrm{Re}\,\mathrm{Pr}\,(\ell _{g_{2}}+\ell _{{\mathbf{v}}_{\mathrm{i}}}\ell _{u_{\mathrm{d}}})\right )+C(\ell _{{\mathbf{v}}_{\mathrm{i}}}+\ell _{u_{\mathrm{d}}}).
\end{split}
\end{equation*}

From the definitions of $K_{1}$ and $K_{2}$, we observe that we can always
find $\tau >0$ and
$\varepsilon _{1},\varepsilon _{2},\varepsilon _{3}>0$ such that if  {\eqref{Stationary-RePrGr}} holds true, then
\begin{align*}
K_{1}<\,1\qquad \text{ and } \qquad K_{2}\leq \, \tau ,
\end{align*}
provided that $\ell _{{\mathbf{v}}_{\mathrm{i}}}$ and
$\ell _{u_{\mathrm{d}}}$ are sufficiently small. Hence, $F$ is a contraction
from $B_{\tau }$ into itself. Then, by Banach's fixed point theorem there
is a unique solution to $(\mathbb{P}_{\mathrm{w}})$ in the closed ball
$B_{\tau }$.
\end{proof}

\subsubsection{Existence of solutions for $\beta $ satisfying only $\mathrm{(A1)}$}
\label{sec:beta-A1}

We now aim to relax the hypotheses on $\beta $ and still retain existence of
solutions. The proof plan considers stability properties of solutions in
the case $\beta \equiv 0$ in  {Lemma~\ref{Stationary-Lemma}}, and then provides
existence by Leray-Schauder's theorem via compactness properties for the
case $\beta \not\equiv 0$ in  {Theorem~\ref{Stationary-Theorem2}}. We start
with the aforementioned lemma.

\begin{lemma}%
\label{Stationary-Lemma}
Let $\beta \equiv 0$ and $\{g_{2}^{n}\}$ be a sequence in $V_{2}'$ such
that $g_{2}^{n}\to g_{2}$ in $V_{2}'$ as $n\to\infty$. Then, provided that
$\|{\mathbf{v}}_{\mathrm{i}}\|_{W^{1/2,2}(\Gamma _{\mathrm{i}})^{d}}$ and
$\|u_{\mathrm{d}}\|_{W^{1/2,2}(\Gamma _{\mathrm{d}})}$ are small enough,
there exist $\varepsilon _{1}, \varepsilon _{2},\varepsilon _{3}>0$ such
that if  {\eqref{Stationary-RePrGr}} holds true, problems
$(\mathbb{P}_{\mathrm{w}})$ associated to $g_{2}$ or $g_{2}^{n}$,
$n\in \mathbb{N}$, have unique solutions $(\bar {\mathbf{v}},\bar {u})$ and
$(\bar {\mathbf{v}}_{n},\bar {u}_{n})$, respectively, in some closed ball
$B_{\tau }$. Furthermore,
%
\begin{equation}
\label{eq:vustrong}
(\bar {\mathbf{v}}_{n}, \bar {u}_{n})\to (\bar {\mathbf{v}}, \bar {u}) \quad
\text{ in }\quad V_{1}\times V_{2}.
\end{equation}
\end{lemma}

\begin{proof}
The existence and uniqueness of solutions is proved exactly as in  {Theorem~\ref{Stationary-Theorem1}}, but replacing $\ell _{g_{2}}$ by
$\sup _{n}\,\|g_{2}^{n}\|_{V_{2}'}$ in the definition of $K_{2}$.

The proof of  {\eqref{eq:vustrong}} is split into two steps. First, we show
that the sequence of solutions $\{(\bar {\mathbf{v}}_{n},\bar {u}_{n})\}$ converges
weakly in $V_{1}\times V_{2}$ to $(\bar {\mathbf{v}},\bar {u})$, and then we
prove that the sequence of norms
$\{\|(\bar {\mathbf{v}}_{n},\bar {u}_{n})\|_{V_{1}\times V_{2}}\}$ converges
to $\|(\bar {\mathbf{v}},\bar {u})\|_{V_{1}\times V_{2}}$.

\emph{First step}. Since
$\{(\bar {\mathbf{v}}_{n},\bar {u}_{n})\}\subset B_{\tau }$, we have that
$\{(\bar {\mathbf{v}}_{n},\bar {u}_{n})\}$ is bounded in
$V_{1}\times V_{2}$ and hence it admits a weakly convergent subsequence
to some $(\tilde{{\mathbf{v}}},\tilde{u})\in B_{\tau }$. We now prove that
$(\tilde{\mathbf{v}},\tilde{u})$ solves problem $(\mathbb{P}_{\mathrm{w}})$ associated
with $g_{2}$ by taking the limit as $n\to \infty $ on
%
\begin{align}
\label{Stationary-Eq1-Beta0}
\langle A_{1}(\bar {\mathbf{v}}_{n}+{\mathbf{V}},\bar {u}_{n}+U),{\mathbf{y}}
\rangle +\langle B_{1}(\bar {\mathbf{v}}_{n}+{\mathbf{V}},\bar {\mathbf{v}}_{n}+{
\mathbf{V}}),{\mathbf{y}}\rangle =&\,\langle {\mathbf{g}}_{1},{\mathbf{y}}\rangle ,
\\
\label{Stationary-Eq2-Beta0}
\langle A_{2}(\bar {u}_{n}+U),y\rangle +\langle B_{2}(\bar {\mathbf{v}}_{n}+{
\mathbf{V}},\bar {u}_{n}+U,y\rangle =&\,\langle g_{2}^{n},y\rangle ,
\end{align}
for each ${\mathbf{y}}\in V_{1}$ and $y\in V_{2}$. Thus, by the uniqueness of
solutions in $B_{\tau }$, we will have that
$(\tilde{\mathbf{v}},\tilde{u})=(\bar {\mathbf{v}},\bar {u})$ and hence
$(\bar {\mathbf{v}}_{n},\bar {u}_{n})\rightharpoonup (\bar {\mathbf{v}},
\bar {u})$ in $V_{1}\times V_{2}$ along the entire sequence (and not only
a subsequence).

Let ${\mathbf{y}}\in V_{1}$ and $y\in V_{2}$. From the weak convergences
$\bar {\mathbf{v}}_{n}\rightharpoonup \bar {\mathbf{v}}$ in $V_{1}$,
$\bar {u}_{n}\rightharpoonup \bar {u}$ in $V_{2}$, and the strong convergence
$g_{2}^{n}\to g_{2}$ in $V_{2}'$, we obtain that
\begin{align*}
\langle A_{1}(\bar {\mathbf{v}}_{n}+{\mathbf{V}},\bar {u}_{n}+U),{\mathbf{y}}
\rangle \quad \to \quad & \langle A_{1}(\bar {\mathbf{v}}+{\mathbf{V}},\bar {u}+U),{
\mathbf{y}}\rangle ,
\\
\langle A_{2}(\bar {u}_{n}+U),y\rangle \quad \to \quad & \langle A_{2}(
\bar {u}+U),y\rangle ,
\\
\langle g_{2}^{n},y\rangle \quad \to \quad & \langle g_{2},y\rangle .
\end{align*}
We notice that
\begin{equation*}
\langle B_{1}(\bar {\mathbf{v}}_{n}+{\mathbf{V}},\bar {\mathbf{v}}_{n}+{\mathbf{V}},{\mathbf{y}}
\rangle =\displaystyle \sum _{i,j=1}^{d}\int _{\Omega }{\mathbf{y}}_{j}\,(
\bar {\mathbf{v}}_{n}+{\mathbf{V}})_{i}\,
\frac{\partial (\bar {\mathbf{v}}_{n}+{\mathbf{V}})_{j}}{\partial x_{i}}\,
\mathrm{d}x.
\end{equation*}
Let $i,j\in \{1,\ldots ,d\}$. Since
${\bar {\mathbf{v}}}_{n}\rightharpoonup \bar {\mathbf{v}}$ in $V_{1}$, we have
\begin{equation*}
\frac{\partial (\bar {\mathbf{v}}_{n}+{\mathbf{V}})_{j}}{\partial x_{i}}\quad
\rightharpoonup \quad
\frac{\partial (\bar {\mathbf{v}}+{\mathbf{V}})_{j}}{\partial x_{i}}\qquad
\text{in}\quad L^{2}(\Omega ).
\end{equation*}
In addition,
\begin{equation*}
{\mathbf{y}}_{j}\,(\bar {\mathbf{v}}_{n}+{\mathbf{V}})_{i}\quad \to \quad {\mathbf{y}}_{j}
\,(\bar {\mathbf{v}}+{\mathbf{V}})_{i}\qquad \text{in}\quad L^{2}(\Omega ),
\end{equation*}
given that
$(\bar {\mathbf{v}}_{n}+{\mathbf{V}})_{i}\to (\bar {\mathbf{v}}+{\mathbf{V}})_{i}$ in
$L^{4}(\Omega )$ by virtue of the compact embedding
$W^{1,2}(\Omega )\hookrightarrow L^{4}(\Omega )$. Therefore,
%
\begin{align}
\label{Stationary-Convergence-B1-prev}
\langle B_{1}(\bar {\mathbf{v}}_{n}+{\mathbf{V}},\bar {\mathbf{v}}_{n}+{\mathbf{V}}),{
\mathbf{y}}\rangle \quad \to \quad & \langle B_{1}(\bar {\mathbf{v}}+{\mathbf{V}},
\bar {\mathbf{v}}+{\mathbf{V}}),{\mathbf{y}}\rangle .
\end{align}
By an identical argument, we find that
%
\begin{align}
\label{Stationary-Convergence-B1}
\langle B_{2}(\bar {\mathbf{v}}_{n}+{\mathbf{V}},\bar {u}_{n}+U),y\rangle
\quad \to \quad & \langle B_{2}(\bar {\mathbf{v}}+{\mathbf{V}},\bar {u}+U),y
\rangle ,
\end{align}
which completes the proof of the first step, i.e.,
$(\bar {\mathbf{v}}_{n},\bar {u}_{n})\rightharpoonup (\bar {\mathbf{v}},
\bar {u})$ in $V_{1}\times V_{2}$.

\emph{Second step}. Taking ${\mathbf{y}}=\bar {\mathbf{v}}_{n}$ and
$y=\bar {u}_{n}$ in  {\eqref{Stationary-Eq1-Beta0}} and  {\eqref{Stationary-Eq2-Beta0}}, respectively, we find that
%
\begin{align}
\label{Stationary-Eq1-Beta0-n}
\frac{1}{\mathrm{Re}}\|\bar {\mathbf{v}}_{n}\|_{V_{1}}^{2}=\,&\langle {
\mathbf{g}}_{1},\bar {\mathbf{v}}_{n}\rangle - \langle B_{1}(\bar {\mathbf{v}}_{n}+{
\mathbf{V}},\bar {\mathbf{v}}_{n}+{\mathbf{V}}),\bar {\mathbf{v}}_{n}\rangle -\langle A_{1}({
\mathbf{V}},\bar {u}_{n}+U),\bar {\mathbf{v}}_{n}\rangle ,
\\
\label{Stationary-Eq2-Beta0-n}
\frac{1}{\mathrm{Re}\,\mathrm{Pr}}\|\bar {u}_{n}\|_{V_{2}}^{2}=\,&\langle g_{2},
\bar {u}_{n}\rangle - \langle B_{2}(\bar {\mathbf{v}}_{n}+{\mathbf{V}},\bar {u}_{n}+U),
\bar {u}_{n}\rangle -\langle A_{2}(U),\bar {u}_{n}\rangle .
\end{align}
Similarly, we have
%
\begin{align}
\label{Stationary-Eq1-Beta0-*}
\frac{1}{\mathrm{Re}}\|\bar {\mathbf{v}}\|_{V_{1}}^{2}=\,&\langle {\mathbf{g}}_{1},
\bar {\mathbf{v}}\rangle - \langle B_{1}(\bar {\mathbf{v}}+{\mathbf{V}},\bar {\mathbf{v}}+{
\mathbf{V}}),\bar {\mathbf{v}}\rangle -\langle A_{1}({\mathbf{V}},\bar {u}+U),
\bar {\mathbf{v}}\rangle ,
\\
\label{Stationary-Eq2-Beta0-*}
\frac{1}{\mathrm{Re}\,\mathrm{Pr}}\|\bar {u}\|_{V_{2}}^{2}=\,&\langle g_{2},
\bar {u}\rangle - \langle B_{2}(\bar {\mathbf{v}}+{\mathbf{V}},\bar {u}+U),
\bar {u}\rangle -\langle A_{2}(U),\bar {u}\rangle .
\end{align}
We shall prove the convergence
$\|(\bar {\mathbf{v}}_{n},\bar {u}_{n})\|_{V_{1}\times V_{2}}\to \|(
\bar {\mathbf{v}},\bar {u})\|_{V_{1}\times V_{2}}$ by showing that the right-hand
sides of  {\eqref{Stationary-Eq1-Beta0-n}} and  {\eqref{Stationary-Eq2-Beta0-n}} converge to the right-hand sides of  {\eqref{Stationary-Eq1-Beta0-*}} and  {\eqref{Stationary-Eq2-Beta0-*}}, respectively.

The weak convergences $\bar {\mathbf{v}}_{n}\rightharpoonup \bar {\mathbf{v}}$ in
$V_{1}$, $\bar {u}_{n}\rightharpoonup \bar {u}$ in $V_{2}$, and the strong
convergence $g_{2}^{n}\to g_{2}$ in $V_{2}'$ imply
\begin{align*}
\langle A_{1}({\mathbf{V}},\bar {u}_{n}+U),\bar {\mathbf{v}}_{n}\rangle \quad
\to \quad & \langle A_{1}({\mathbf{V}},\bar {u}+U),\bar {\mathbf{v}}\rangle ,
\\
\langle A_{2}(U),\bar {u}_{n}\rangle \quad \to \quad & \langle A_{2}(U),
\bar {u}\rangle ,
\\
\langle {\mathbf{g}}_{1},\bar {\mathbf{v}}_{n}\rangle \quad \to \quad &\langle {
\mathbf{g}}_{1},\bar {\mathbf{v}}\rangle ,
\\
\langle g_{2}^{n},\bar {u}_{n}\rangle \quad \to \quad &\langle g_{2},
\bar {u}\rangle .
\end{align*}
To obtain the first convergence, we have used that
$\bar {u}_{n}\to \bar {u}$ by virtue of the compact embedding
$W^{1,2}(\Omega )\hookrightarrow L^{2}(\Omega )$. Analogously as we proved  {\eqref{Stationary-Convergence-B1-prev}} and  {\eqref{Stationary-Convergence-B1}}, we obtain that
\begin{equation*}
\begin{split} \langle B_{1}(\bar {\mathbf{v}}_{n}+{\mathbf{V}},\bar {\mathbf{v}}_{n}+{
\mathbf{V}}),\bar {\mathbf{v}}_{n}\rangle \to \,& \langle B_{1}(\bar {\mathbf{v}}+{
\mathbf{V}},\bar {\mathbf{v}}+{\mathbf{V}},\bar {\mathbf{v}}\rangle ,
\\
\langle B_{2}(\bar {\mathbf{v}}_{n}+{\mathbf{V}},\bar {u}_{n}+U),\bar {u}_{n}
\rangle \to \,&\langle B_{2}(\bar {\mathbf{v}}+{\mathbf{V}},\bar {u}+U,\bar {u}
\rangle ,
\end{split}
\end{equation*}
and hence
$\|(\bar {\mathbf{v}}_{n},\bar {u}_{n})\|_{V_{1}\times V_{2}} \to \|(
\bar {\mathbf{v}},\bar {u})\|_{V_{1}\times V_{2}}$, which completes the proof.
\end{proof}

We are now in shape to prove the existence result for the case when
$\beta $ only satisfies (A1) but is not necessarily Lipschitz by making
use on the above lemma.

\begin{theorem}%
\label{Stationary-Theorem2}
Provided that
$\|{\mathbf{v}}_{\mathrm{i}}\|_{W^{1/2,2}(\Gamma _{\mathrm{i}})^{d}}$ and
$\|u_{\mathrm{d}}\|_{W^{1/2,2}(\Gamma _{\mathrm{i}})}$ are sufficiently small,
there exist $\varepsilon _{1},\varepsilon _{2},\varepsilon _{3}>0$ such
that if  {\eqref{Stationary-RePrGr}} holds true, then problem
$(\mathbb{P}_{\mathrm{w}})$ admits at least one solution.
\end{theorem}

\begin{proof}
Let $\tau >0$ and $(\tilde{\mathbf{v}},\tilde{u})\in B_{\tau }$. From the estimate  {\eqref{H-W12}} we find that
\begin{equation*}
\|H(\tilde{\mathbf{v}}+{\mathbf{V}},\tilde{u}+U)\|_{V_{2}'}\leq C\,\|\beta \|_{L^{\infty }(\mathbb{R})}(\tau +\|{\mathbf{v}}_{\mathrm{i}}\|_{W^{1,2}(\Gamma _{\mathrm{i}})^{d}})(\tau +\|u_{\mathrm{d}}\|_{W^{1,2}(\Gamma _{\mathrm{d}})}).
\end{equation*}

Thus, we replace $\ell _{g_{2}}$ by
$C\|\beta \|_{L^{\infty }(\mathbb{R})}(\tau +\|{\mathbf{v}}_{\mathrm{i}}\|_{W^{1,2}(
\Gamma _{\mathrm{i}})^{d}})(\tau +\|u_{\mathrm{d}}\|_{W^{1,2}(\Gamma _{\mathrm{d}})})$ in the definition of $K_{2}$ in the proof of  {Theorem~\ref{Stationary-Theorem1}} and observe that the same arguments yield the
existence and uniqueness of a solution
$(\bar {\mathbf{v}},\bar {u})\in B_{\tau }$ to
%
\begin{align}
\label{Stationary-Eq1-FixedPoint}
A_{1}(\bar {\mathbf{v}}+{\mathbf{V}},\bar {u}+U)+B_{1}(\bar {\mathbf{v}}+{\mathbf{V}},
\bar {\mathbf{v}}+{\mathbf{V}})=&\,{\mathbf{g}}_{1}&\text{in}&\quad V_{1}',
\\
\label{Stationary-Eq2-FixedPoint}
A_{2}(\bar {u}+U)+B_{2}(\bar {\mathbf{v}}+{\mathbf{V}},\bar {u}+U)=&\,g_{2}+H(
\tilde{\mathbf{v}}+{\mathbf{V}},\tilde{u}+U)&\text{in}&\quad V_{2}',
\end{align}
for some $\tau >0$. Therefore, the map $G:B_{\tau }\to B_{\tau }$, given by
$G(\tilde{\mathbf{v}},\tilde{u})=(\bar {\mathbf{v}},\bar {u})$, where
$(\bar {\mathbf{v}},\bar {u})\in B_{\tau }$ satisfies  {\eqref{Stationary-Eq1-FixedPoint}} and  {\eqref{Stationary-Eq2-FixedPoint}}, is well-defined. From  {\eqref{Stationary-Eq1-FixedPoint}}- {\eqref{Stationary-Eq2-FixedPoint}}, we
observe that any $(\bar {\mathbf{v}},\bar {u})\in B_{\tau }$ such that
$G(\bar {\mathbf{v}},\bar {u})=(\bar {\mathbf{v}},\bar {u})$ is a solution to
$(\mathbb{P}_{\mathrm{w}})$. The rest of the proof is devoted to showing that
$G$ is a compact map, which will give the existence of solutions to
$(\mathbb{P}_{\mathrm{w}})$ by Leray-Schauder's fixed point theorem.

Let $\{({\mathbf{w}}_{n},w_{n})\}\subset B_{\tau }$ be a weakly convergent sequence
in $V_{1}\times V_{2}$ to some $({\mathbf{w}},w)\in B_{\tau }$. We observe that,
by virtue of  {Lemma~\ref{Stationary-Lemma}}, to prove the convergence
$G({\mathbf{w}}_{n},w_{n})\to G({\mathbf{w}},w)$ it suffices to show that
%
\begin{equation}
\label{Stationary-Convegence-H}
H({\mathbf{w}}_{n},w_{n})\to H({\mathbf{w}},w)\quad \text{in}\quad V_{2}'.
\end{equation}

Since the embedding
$W^{1,2}(\Omega )\hookrightarrow L^{p}(\Gamma _{\mathrm{o}})$ is compact
for any $1\leq p<4$ (see Theorem 6.2 of \cite{Ne2012}), we find that
%
\begin{equation}
\label{Stationary-Convergences-uv}
{\mathbf{w}}_{n}\to {\mathbf{w}}\quad \text{in}\quad L^{p}(\Gamma _{\mathrm{o}})^{d}
\quad \text{ and } \quad w_{n}\to w\quad \text{in}\quad L^{p}(\Gamma _{\mathrm{o}})\quad \text{for}\quad 1\leq p<4.
\end{equation}
We notice that the map
$\vartheta :\Gamma _{\mathrm{o}}\times \mathbb{R}^{d}\to \mathbb{R}$ given
by $\vartheta (x,z):=\beta (z\cdot {\mathbf{n}}(x))\,(z\cdot {\mathbf{n}}(x))$ is continuous
with respect to the second argument for a.e.
$x\in \Gamma _{\mathrm{o}}$ since the only point where $\beta $ is allowed
to be discontinuous is the origin. Furthermore, $\vartheta $ satisfies
the growth condition
$|\vartheta (x,z)|\leq \|\beta \|_{L^{\infty }(\mathbb{R})}|z|$. Therefore,
the Nemytskii operator
\begin{equation*}
\tilde{\mathbf{w}}\mapsto F_{\beta }(\tilde{\mathbf{w}}):=\beta (\tilde{\mathbf{w}}
\cdot {\mathbf{n}})(\tilde{\mathbf{w}}\cdot {\mathbf{n}})
\end{equation*}
generated by $\vartheta $ is continuous from
$L^{p}(\Gamma _{\mathrm{o}})^{d}$ into $L^{p}(\Gamma _{\mathrm{o}})$ (see
Theorems 1 and 4 of \cite{GoKaTr1992}). Combining this with the first convergence
in  {\eqref{Stationary-Convergences-uv}}, we obtain that
\begin{equation*}
F_{\beta }({\mathbf{w}}_{n}+{\mathbf{V}})\to F_{\beta }({\mathbf{w}}+{\mathbf{V}})\quad
\text{in}\quad L^{p}(\Gamma _{\mathrm{o}})\quad 1\leq p<4.
\end{equation*}
From this and the second convergence in  {\eqref{Stationary-Convergences-uv}}, we find that
%
\begin{equation}
\label{Stationary-Convergence-beta}
(w_{n}+U)\,F_{\beta }({\mathbf{w}}_{n}+{\mathbf{V}})\to (w+U)\,F_{\beta }({\mathbf{w}}+{
\mathbf{V}}) \quad \text{in}\quad L^{p}(\Gamma _{\mathrm{o}})\qquad 1\leq p<2.
\end{equation}

Let $\xi \in (0,2/3)$ and $y\in V_{2}$. We observe that $p:=2-\xi $ and
its H\"{o}lder's conjugate $p'=\frac{p}{p-1}$ satisfy $1\leq p <2$ and
$1\leq p'<4$. Then, by H\"{o}lder's inequality we obtain that
\begin{equation*}
\begin{split} &\displaystyle \int _{\Gamma _{\mathrm{o}}} \left |(w_{n}+U)
\,F_{\beta }({\mathbf{w}}_{n}+{\mathbf{V}})-(w+U)\,F_{\beta }({\mathbf{w}}+{\mathbf{V}})
\right |\,|y|\,\mathrm{d}\sigma
\\
&\qquad \leq ||(w_{n}+U)\,F_{\beta }({\mathbf{w}}_{n}+{\mathbf{V}})-(w+U)\,F_{\beta }({\mathbf{w}}+{\mathbf{V}})||_{L^{p}(\Gamma _{\mathrm{o}})}\,||y||_{L^{p'}(
\Gamma _{\mathrm{o}})},
\end{split}
\end{equation*}
which, together with  {\eqref{Stationary-Convergence-beta}}, gives  {\eqref{Stationary-Convegence-H}}.
\end{proof}

\section{The evolutionary problem}\label{sec:evo}
In this section we consider the same geometry and boundary conditions established
in section~\ref{Stationary}, and extend results into the time evolutionary
setting. The task is significantly more complex than in the stationary
case: In this case an appropriate state space for solutions is tailored
specifically for our special case, and a restrictive partition of the boundary
is required. The problem of interest is the following time evolutionary
one.

\bigskip\noindent
\textbf{Problem $(\tilde{\mathbb{P}})$}:
Let $T>0$ be given. \emph{Find
${\mathbf{v}}:(0,T)\times \Omega \to \mathbb{R}^{d}$,
$u:(0,T)\times \Omega \to \mathbb{R}$, and
$p:(0,T)\times \Omega \to \mathbb{R}$ that satisfy the evolutionary Boussinesq
equations in} $(0,T)\times \Omega $,
%
\begin{align}
\label{BoussinesqEvolutionary}
\frac{\partial {\mathbf{v}}}{\partial t}+{\mathbf{v}}\cdot \nabla {\mathbf{v}}-
\frac{1}{\mathrm{Re}}\Delta {\mathbf{v}}+\nabla p=\,&\frac{\mathrm{Gr}}{\mathrm{Re}^{2}}
\,u\,{\mathbf{e}}+{\mathbf{g}}_{1},
\\
\label{BoussinesqEvolutionary2}
\mathrm{div}\,{{\mathbf{v}}}=\,&0,
\\
\label{BoussinesqEvolutionary3}
\frac{\partial u}{\partial t}+{\mathbf{v}}\cdot \nabla u-
\frac{1}{\mathrm{Re}\,\mathrm{Pr}}\Delta u=\,&g_{2},
\end{align}
\emph{subject to the boundary conditions}
%
\begin{align}
\label{BoundaryConditionsEvolutionary}
{\mathbf{v}}&={\mathbf{v}}_{\mathrm{i}}\quad \,\text{\emph{on}}\quad (0,T)\,\times \,
\Gamma _{\mathrm{i}},&\quad {\mathbf{v}}&=0\quad \,\,\,\text{\emph{on}}\quad (0,T)
\,\times \,\Gamma _{\mathrm{w}},
\\
\frac{\partial u}{\partial {\mathbf{n}}}&=0\quad \,\,\,\,\text{\emph{on}}\quad (0,T)
\,\times \,\Gamma _{\mathrm{n}},&\quad u&=u_{\mathrm{d}}\quad \text{\emph{on}}
\quad (0,T)\,\times \,\Gamma _{\mathrm{d}},
\\
\frac{1}{\mathrm{Re}}\,\frac{\partial {\mathbf{v}}}{\partial {\mathbf{n}}}&=p\,{
\mathbf{n}}\quad \text{\emph{on}}\quad (0,T)\,\times \,\Gamma _{\mathrm{o}},&\quad
\frac{1}{\mathrm{Re}\,\mathrm{Pr}}\,
\frac{\partial u}{\partial {\mathbf{n}}}&= u\,\beta ({\mathbf{v}}\cdot {\mathbf{n}})\,({
\mathbf{v}}\cdot {\mathbf{n}})\quad \text{\emph{on}}\quad (0,T)\,\times \,\Gamma _{\mathrm{o}},
\end{align}
\emph{and the initial conditions}
%
\begin{equation}
\label{InitialConditions}
\begin{array}{l}
{\mathbf{v}}={\mathbf{v}}_{\mathrm{o}},\qquad u=u_{\mathrm{o}},\qquad p=p_{\mathrm{o}}\qquad
 \text{\emph{on}}\quad \{0\}\times \overline{\Omega }.
\end{array}
\end{equation}
Here ${\mathbf{g}}_{1}:(0,T)\times \Omega \to \mathbb{R}^{d}$ and 
$g_{2}:(0,T)\times \Omega \to \mathbb{R}$ are given together with
${\mathbf{v}}_{\mathrm{o}}:\Omega \to \mathbb{R}^{d}$,
$u_{\mathrm{o}}:\Omega \to \mathbb{R}$, and
$p_{\mathrm{o}}:\Omega \to \mathbb{R}$, which represent the initial velocity,
temperature, and pressure distributions, respectively.

\subsection{Weak formulation}
\label{sec:weakformu-evo}

In order to address the time evolutionary problem, we define the spaces
\begin{equation*}
W_{i}(0,T):=\{\varphi \in L^{2}(0,T;V_{i}):\partial _{t}\,\varphi
\in L^{2}(0,T;V_{i}')\}
\end{equation*}
for $i=1,2$, endowed with the norms
\begin{equation*}
\|\varphi \|_{W_{i}(0,T)}:=\|\varphi \|_{L^{2}(0,T;V_{i})}+\|
\partial _{t}\,\varphi \|_{L^{2}(0,T;V_{i}')}.
\end{equation*}
The weak formulation of problem $(\tilde{\mathbb{P}})$ is determined following
analogous steps as for the stationary case leading to:

\bigskip\noindent
{\textbf{Problem $(\tilde{\mathbb{P}}_{\mathrm{w}})$}}: \emph{Find
$(\bar {\mathbf{v}},\bar {u})\in W_{1}(0,T)\times W_{2}(0,T)$ such that
$\bar {\mathbf{v}}(0)={\mathbf{v}}_{\mathrm{o}}-{\mathbf{V}}$,
$\bar {u}(0)=u_{\mathrm{o}}-U$, and}
\begin{align*}
\partial _{t}\bar {\mathbf{v}}(t)+A_{1}(\bar {\mathbf{v}}(t)+{\mathbf{V}},\bar {u}(t)+U)+B_{1}(
\bar {\mathbf{v}}(t)+{\mathbf{V}},\bar {\mathbf{v}}(t)+{\mathbf{V}})=\,&{\mathbf{g}}_{1}(t)
\qquad &\text{in}&\quad V_{1}',
\\
\partial _{t}\bar {u}(t)+A_{2}(\bar {\mathbf{v}}(t)+{\mathbf{V}},\bar {u}(t)+U)+B_{2}(
\bar {\mathbf{v}}(t)+{\mathbf{V}},\bar {u}(t)+U)&&&%
\\
-H(\bar {\mathbf{v}}(t)+{\mathbf{V}},\bar {u}(t)+U)=\,&g_{2}(t)\qquad &\text{in}&\quad V_{2}',
\end{align*}
\emph{for a.e. $t\in (0,T)$, where ${\mathbf{g}}_{1}\in L^{2}(0,T;V_{1}')$,
$g_{2}\in L^{2}(0,T;V_{2}')$,
${\mathbf{v}}_{\mathrm{o}}\in L^{2}(\Omega )^{d}$, and}
$u_{\mathrm{o}}\in L^{2}(\Omega )$.

\bigskip\noindent
Then, a weak solution to $(\tilde{\mathbb{P}})$ is defined as
$(\bar {\mathbf{v}}+{\mathbf{V}},\bar {u}+U)$ where $(\bar {\mathbf{v}},\bar {u})$ solves
$(\tilde{\mathbb{P}}_{\mathrm{w}})$. Further, notice that the initial conditions
on $\bar {\mathbf{v}}$ and $\bar {u}$ in problem
$(\tilde{\mathbb{P}}_{\mathrm{w}})$ are meaningful since
$W_{i}(0,T)\hookrightarrow C([0,T];H_{i})$ for $i=1,2$ (see, e.g.,
\cite{Sh1997}).

\subsection{Existence of weak solutions}
\label{sec:weakformuexis-evo}

The existence and uniqueness of a weak solution to
$(\tilde{\mathbb{P}})$ will be proven in  {Theorem~\ref{Evolutionary-Theorem1}} in the case that $\beta $ is a Lipschitz continuous
function and $d=2$. The proof is analogous to the one of section~\ref{Stationary-WeakSolutions} but where state spaces are defined in a
non-trivial fashion, and we restrict boundary parts to have specific geometrical
properties. In particular, we do not require working with Sobolev spaces
of non-integer order, which are sometimes used to show existence of weak
solutions to Boussinesq systems with mixed boundary conditions in open
domains; see \cite{Be2014,BeTi2015,SkKu2000}. Like in the stationary case,
the proof is built upon the assumption of small data, and provides existence
and uniqueness for $\mathrm{Re}$ and $\mathrm{Gr}$ sufficiently small,
and for $\mathrm{Pr}$ possibly large. The sense, in which all of this is
considered, is given in the statement of  {Theorem~\ref{Evolutionary-Theorem1}}.

We start with the following lemma which provides continuity properties
on some specific Bochner spaces for $B_{1}$ and $B_{2}$ in the time dependent
case.

\begin{lemma}%
\label{Evolutionary-Lemma0}
Let
${\mathbf{w}},{\mathbf{w}}_{1},{\mathbf{w}}_{2}\in L^{2}(0,T;W^{2,2}(\Omega )^{d})
\cap L^{\infty }(0,T;W^{1,2}(\Omega )^{d})$ and
$w\in L^{2}(0,T;W^{1,2}(\Omega ))$ be arbitrary. Then, we observe:

\textrm{1.} The maps
$B_{1}({\mathbf{w}}_{1},{\mathbf{w}}_{2}):(0,T)\to V_{1}'$ and
$B_{2}({\mathbf{w}},w):(0,T)\to V_{2}'$, given by
\begin{equation*}
B_{1}({\mathbf{w}}_{1},{\mathbf{w}}_{2})(t):=B_{1}({\mathbf{w}}_{1}(t),{\mathbf{w}}_{2}(t)),
\qquad B_{2}({\mathbf{w}},w)(t):=B_{2}({\mathbf{w}}(t),w(t)),
\end{equation*}
belong to $L^{2}(0,T;L^{2}(\Omega )^{d})$ and $L^{2}(0,T;V_{2}')$, respectively,
and satisfy the estimates
%
\begin{equation}
\label{Evolutionary-B1-Bochner}
\|B_{1}({\mathbf{w}}_{1},{\mathbf{w}}_{2})\|_{L^{2}(0,T;L^{2}(\Omega )^{d})}
\leq \, C \|{\mathbf{w}}_{1}\|_{L^{2}(0,T;W^{2,2}(\Omega )^{d})}\|{\mathbf{w}}_{2}
\|_{L^{\infty }(0,T;W^{1,2}(\Omega )^{d})}
\end{equation}
and
%
\begin{equation}
\label{Evolutionary-B2-Bochner}
\|B_{2}({\mathbf{w}},w)\|_{L^{2}(0,T;V_{2}')}\leq \, C \|{\mathbf{w}}\|_{L^{\infty }(0,T;W^{1,2}(\Omega )^{d})}\|w\|_{L^{2}(0,T;W^{1,2}(\Omega ))}.
\end{equation}

\textrm{2.} The map $H({\mathbf{w}},w):(0,T)\to V_{2}'$, given by
\begin{equation*}
H({\mathbf{w}},w)(t):=H({\mathbf{w}}(t),w(t)),
\end{equation*}
belongs to $L^{2}(0,T;V_{2}')$ and satisfies the estimate
\begin{equation*}
\|H({\mathbf{w}},w)\|_{L^{2}(0,T;V_{2}')}\leq C\,\|\beta \|_{L^{\infty }(
\mathbb{R})}\|{\mathbf{w}}\|_{L^{\infty }(0,T;W^{1,2}(\Omega )^{d})}\|w\|_{L^{2}(0,T;W^{1,2}(
\Omega ))}.
\end{equation*}
\end{lemma}

\begin{proof}
\emph{1.} Let $t\in (0,T)$. From the embedding
$W^{2,2}(\Omega )\hookrightarrow L^{\infty }(\Omega )$, we find that
%
\begin{equation}
\label{Evolutionary-B1}
\begin{split} \int _{\Omega }|{\mathbf{w}}_{1}(t)\cdot \nabla {\mathbf{w}}_{2}(t)|^{2}
\,\mathrm{d}x\leq \,& \|{\mathbf{w}}_{1}(t)\|_{L^{\infty }(\Omega )^{d}}^{2}
\|\nabla {\mathbf{w}}_{2}(t)\|^{2}_{L^{2}(\Omega )^{d\times d}}
\\
\leq \,& C\,\|{\mathbf{w}}_{1}(t)\|_{W^{2,2}(\Omega )^{d}}^{2}\|{\mathbf{w}}_{2}(t)
\|_{W^{1,2}(\Omega )^{d}}^{2}.
\end{split}
\end{equation}
Hence, ${\mathbf{w}}_{1}(t)\cdot \nabla {\mathbf{w}}_{2}(t)\in L^{2}(\Omega )$, and
integrating with respect to $t$ from $0$ to $T$ the above expression, we
observe
\begin{equation*}
\begin{split} \int _{0}^{T}\|{\mathbf{w}}_{1}(t)\cdot \nabla {\mathbf{w}}_{2}(t)
\|^{2}_{L^{2}(\Omega )^{d}}\,\mathrm{d}t\leq \,&C\int _{0}^{T}\|{
\mathbf{w}}_{1}(t)\|_{W^{2,2}(\Omega )^{d}}^{2}\|{\mathbf{w}}_{2}(t)\|_{W^{1,2}(
\Omega )^{d}}^{2}\,\mathrm{d}t
\\
\leq \,&C\,\|{\mathbf{w}}_{1}\|^{2}_{L^{2}(0,T;W^{2,2}(\Omega )^{d})}\|{
\mathbf{w}}_{2}\|^{2}_{L^{\infty }(0,T;W^{1,2}(\Omega )^{d})}.
\end{split}
\end{equation*}
Therefore,
$B_{1}({\mathbf{w}}_{1},{\mathbf{w}}_{2})\in L^{2}(0,T;L^{2}(\Omega )^{d})$ with
the desired estimate. The result concerning the map
$B_{2}({\mathbf{w}},w)$ is obtained similarly from the estimate  {\eqref{B2-W12}}.

\emph{2.} Integrating the square of  {\eqref{H-W12}} with respect to
$t$ between $0$ and $T$, we obtain that
\begin{equation*}
\begin{split} \int _{0}^{T}\|H({\mathbf{w}},w)(t)\|_{V_{2}'}^{2}\,
\mathrm{d}t \leq \,&C\,\|\beta \|_{L^{\infty }(\mathbb{R})}^{2}\|{\mathbf{w}}
\|_{L^{\infty }(0,T;W^{1,2}(\Omega )^{d})}^{2}\|w\|_{L^{2}(0,T;W^{1,2}(
\Omega ))}^{2}.
\end{split}
\end{equation*}
Hence, $H({\mathbf{w}},w)\in L^{2}(0,T;V_{2}')$ with the desired estimate.
\end{proof}

Next we provide the result of existence and uniqueness of solutions to
the Boussinesq system of interest. The proof is given for the case when
$d=2$ and the boundary subparts $\Gamma _{\mathrm{i}}$,
$\Gamma _{\mathrm{w}}$, $\Gamma _{\mathrm{o}}$ satisfy additional regularity
and geometrical properties, which allow to obtain weak solutions with
increased spatial regularity.

\begin{theorem}%
\label{Evolutionary-Theorem1}
Let $d=2$ and $\beta :\mathbb{R}\to \mathbb{R}$ be Lipschitz continuous.
Assume that
$\Gamma _{\mathrm{i}}\cup \Gamma _{\mathrm{w}}\neq \emptyset $, the set
$A:=\Gamma \setminus (\Gamma _{\mathrm{i}}\cup \Gamma _{\mathrm{w}}
\cup \Gamma _{\mathrm{o}})$ is finite, any portion of
$\Gamma _{\mathrm{o}}$ is flat and forms a right angle with
$\Gamma _{\mathrm{i}}\cup \Gamma _{\mathrm{w}}$ at each contact point in
$A$, and that any boundary subpart of $\Gamma _{\mathrm{i}}$ or
$\Gamma _{\mathrm{w}}$ is of class $C^{\infty }$. Further, suppose that
\begin{equation*}
{\mathbf{g}}_{1}\in L^{2}(0,T;L^{2}(\Omega )^{2}),\qquad {\mathbf{v}}_{\mathrm{o}}\in V_{1},\qquad \text{and}\qquad {\mathbf{v}}_{\mathrm{i}}
\equiv 0.
\end{equation*}
Then, provided that
\begin{equation*}
\|{\mathbf{g}}_{1}\|_{L^{2}(0,T;L^{2}(\Omega )^{2})},\qquad \|g_{2}\|_{L^{2}(0,T;V_{2}')},
\qquad \|{\mathbf{v}}_{\mathrm{o}}\|_{V_{1}},\qquad \|u_{\mathrm{o}}\|_{L^{2}(
\Omega )},\qquad \|u_{\mathrm{d}}\|_{W^{1/2,2}(\Gamma _{\mathrm{d}})}
\end{equation*}
are sufficiently small, there exist
$\varepsilon _{1},\varepsilon _{2},\varepsilon _{3}>0$ such that if
%
\begin{equation}
\label{A5}
\mathrm{Re}\in (0,\varepsilon _{1}),\qquad \mathrm{Pr}\in \left (0,
\frac{\varepsilon _{2}}{\mathrm{Re}}\right ),\qquad \mathrm{Gr}\in (0,
\varepsilon _{3}\,\min (\mathrm{Re},\mathrm{Re}^{2})),
\end{equation}
there exists a unique solution
$(\bar {\mathbf{v}},\bar {u})\in W_{1}(0,T)\times W_{2}(0,T)$ to problem
$(\tilde{\mathbb{P}}_{\mathrm{w}})$ satisfying
\begin{equation*}
\bar {\mathbf{v}}\in \tilde{W}_{1}(0,T),
\end{equation*}
where $ \tilde{W}_{1}(0,T)\subset W_{1}(0,T)$ is given by
\begin{equation*}
\tilde{W}_{1}(0,T):=\,\{{\mathbf{w}}\in L^{2}(0,T;W^{2,2}(\Omega )^{2})
\cap L^{\infty }(0,T;V_{1}):\partial _{t}{\mathbf{w}}\in L^{2}(0,T;L^{2}(
\Omega )^{2})\}.
\end{equation*}
\end{theorem}

\begin{proof}
First note that since ${\mathbf{v}}_{\mathrm{i}}\equiv 0$ by assumption,
we can select ${\mathbf{V}}\equiv 0$ also. Hence, we can write the equations
in problem $(\tilde{\mathbb{P}}_{\mathrm{w}})$ as follows: For a.e.
$t\in (0,T)$, we have
\begin{align*}
\langle \partial _{t}\bar {\mathbf{v}}(t),{\mathbf{w}}\rangle +
\frac{1}{\mathrm{Re}}(\nabla \bar {\mathbf{v}}(t),\nabla {\mathbf{w}})_{2}=\,&({
\mathbf{g}}_{1}(t),{\mathbf{w}})_{2}-\langle B_{1}(\bar {\mathbf{v}}(t),\bar {\mathbf{v}}(t)),{
\mathbf{w}}\rangle
\\
&-\langle A_{1}(0,\bar {u}(t)+U),{\mathbf{w}}\rangle \qquad \quad \forall
\,{\mathbf{w}}\in V_{1},
\\
\langle \partial _{t}\bar {u}(t),w\rangle
\hspace*{-0.05cm}
+
\hspace*{-0.05cm}
\frac{1}{\mathrm{Re}\,\mathrm{Pr}}(\nabla \bar {u}(t),\nabla w)_{2}=
\,&\langle g_{2}(t),w\rangle
\hspace*{-0.05cm}
-
\hspace*{-0.05cm}
\langle B_{2}(\bar {\mathbf{v}}(t),\bar {u}(t)
\hspace*{-0.05cm}
+
\hspace*{-0.05cm}
U),w\rangle
\hspace*{-0.05cm}
-
\hspace*{-0.05cm}
\langle A_{2}(U),w\rangle
\\
&%
\hspace*{-0.05cm}%
+
\hspace*{-0.05cm}
\langle H(\bar {\mathbf{v}}(t),\bar {u}(t)+U),w\rangle \qquad \,\,\,\,
\forall \,w\in V_{2}.
\end{align*}
In what follows, we address the existence and uniqueness
of a solution to $(\tilde{\mathbb{P}}_{\mathrm{w}})$ as a fixed point problem
by choosing the state space as $\tilde{W}_{1}(0,T)\times W_{2}(0,T)$.

We endow $\tilde{W}_{1}(0,T)$ with the norm
\begin{equation*}
\|{\mathbf{w}}\|_{\tilde{W}_{1}(0,T)}:=\|{\mathbf{w}}\|_{L^{2}(0,T;W^{2,2}(
\Omega )^{2})}+\|{\mathbf{w}}\|_{L^{\infty }(0,T;V_{1})}+\|\partial _{t}{
\mathbf{w}}\|_{L^{2}(0,T;L^{2}(\Omega )^{2})},
\end{equation*}
and the product space $\tilde{W}_{1}(0,T)\times W_{2}(0,T)$ is considered
with the usual norm
\begin{equation*}
\|({\mathbf{w}},w)\|_{\tilde{W}_{1}(0,T)\times W_{2}(0,T)}:=\|{\mathbf{w}}\|_{
\tilde{W}_{1}(0,T)}+\|w\|_{W_{2}(0,T)}.
\end{equation*}
We define the maps $P_{1}:\tilde{W}_{1}(0,T)\times W_{2}(0,T)\to L^{2}(0,T;L^{2}(
\Omega )^{2})$ and
$P_{2}:\tilde{W}_{1}(0,T)\times W_{2}(0,T)\to L^{2}(0,T;V_{2}')$ by
\begin{equation*}
\begin{split} \langle P_{1}({\mathbf{w}},w)(t),{\mathbf{y}}\rangle :=\,&({\mathbf{g}}_{1}(t),{
\mathbf{y}})_{2}-\langle B_{1}({\mathbf{w}}(t),{\mathbf{w}}(t)),{\mathbf{y}}\rangle -
\langle A_{1}(0,w(t)+U),{\mathbf{y}}\rangle 
\end{split}
\end{equation*}
and
\begin{equation*}
\begin{split} \langle P_{2}({\mathbf{w}},w)(t),y\rangle :=\,&\langle g_{2}(t),y
\rangle -\langle B_{2}({\mathbf{w}}(t),w(t)+U),y\rangle -\langle A_{2}(U),y
\rangle
\\
&+\langle H({\mathbf{w}}(t),w(t)+U),y\rangle .
\end{split}
\end{equation*}
 Note that $P_{1}$ and $P_{2}$ are well-defined by virtue of  {Lemma~\ref{Evolutionary-Lemma0}} and initial assumptions.

In addition, by analogy with the stationary case we define
$S_{1}^{-1}$ as the linear operator that maps any
${\mathbf{h}}\in L^{2}(0,T;L^{2}(\Omega )^{2})$ to the solution
${\mathbf{w}}\in \tilde{W}_{1}(0,T)$ of the abstract evolutionary Stokes problem
\begin{equation*}
\begin{split} &\langle \partial _{t}{\mathbf{w}}(t),{\mathbf{y}}\rangle +
\frac{1}{\mathrm{Re}}(\nabla {\mathbf{w}}(t),\nabla {\mathbf{y}})_{2}=({\mathbf{h}}(t),{
\mathbf{y}})_{2}\quad \text{for all ${\mathbf{y}}\in V_{1}$ and a.e. }t\in (0,T),
\\
&{\mathbf{w}}(0)={\mathbf{v}}_{\mathrm{o}}.
\end{split}
\end{equation*}
Similarly, $S_{2}^{-1}$ maps $h\in L^{2}(0,T;V_{2}')$ into the solution
$w\in W_{2}(0,T)$ to
\begin{equation*}
\begin{split} &\langle \partial _{t}w(t),y\rangle +
\frac{1}{\mathrm{Re}\,\mathrm{Pr}}(\nabla w(t),\nabla y)_{2}=
\langle h_{2}(t),y\rangle \quad \text{for all $y\in V_{2}$ and a.e. }t
\in (0,T),
\\
&w(0)=u_{\mathrm{o}}-U.
\end{split}
\end{equation*}
The existence and uniqueness result that yields the well-defined property
of $S_{1}^{-1}$ can be obtained by analogous arguments of those in Theorem
5 of section 7.1.3 of \cite{Ev2010} together with the regularity result
on solutions to steady Stokes problems given in Theorem A.1 of
\cite{BeKu2016} (see  {Theorem~\ref{Ap}} in the Appendix). The analogous result
for $S_{2}^{-1}$ follows from Proposition III.2.3 of \cite{Sh1997}.

Then, a necessary and sufficient condition for an element
$(\bar {\mathbf{v}},\bar {u})\in \tilde{W}_{1}(0,T)\times W_{2}(0,T)$ to be
a solution to $(\tilde{\mathbb{P}}_{\mathrm{w}})$ is that
\begin{equation*}
F(\bar {\mathbf{v}},\bar {u})=(\bar {\mathbf{v}},\bar {u})
\end{equation*}
for $F:=({S}_{1}^{-1} {P_{1}}, {S}_{2}^{-1}{P_{2}})$.

We denote by $B_{\tau }$ the closed ball in
$\tilde{W}_{1}(0,T)\times W_{2}(0,T)$ with radius $\tau >0$ and center
at the origin. In what follows, we shall prove that $F$ is a contraction
from some closed ball $B_{\tau }$ into itself and hence the existence and
uniqueness of a solution to problem
$(\tilde{\mathbb{P}}_{\mathrm{w}})$ in $B_{\tau }$ will follow by Banach fixed
point theorem.

Initially, we observe that
%
\begin{align}
\label{Evolutionary-S1}
{\|S_{1}^{-1}{\mathbf{h}}\|_{\tilde{W}_{1}(0,T)}}\leq \,& c_{1}(
\mathrm{Re})\|{\mathbf{h}}\|_{L^{2}(0,T;L^{2}(\Omega )^{2})}+c_{2}(
\mathrm{Re}){\|{\mathbf{v}}_{\mathrm{o}}\|_{V_{1}}},
\end{align}
for all ${\mathbf{h}}\in L^{2}(0,T;L^{2}(\Omega )^{2})$, and that
%
\begin{align}
\label{Evolutionary-S2}
{\|S_{2}^{-1}h\|_{W_{2}(0,T)}}&\leq C(1+\mathrm{Re}\,\mathrm{Pr})(
\|h\|_{L^{2}(0,T;V_{2}')}+{\|u_{\mathrm{o}}-U\|_{L^{2}(\Omega )}}),
\end{align}
for all $h\in L^{2}(0,T;V_{2}')$, where
\begin{align*}
c_{1}(\mathrm{Re}):=\,& C(1+\mathrm{Re}^{1/2}+\mathrm{Re})\qquad
\text{and}\qquad c_{2}(\mathrm{Re}):=\,C(1 +\mathrm{Re}^{-1/2}+
\mathrm{Re}^{1/2}).
\end{align*}
The first estimate is proved in  {Theorem~\ref{Ap}}. The second one is a direct
consequence of Proposition III.2.3 of \cite{Sh1997}.

Let $\tau >0$ and consider arbitrary
$({\mathbf{w}}_{1},w_{1}), ({\mathbf{w}}_{2},w_{2})\in B_{\tau }$. Estimates  {\eqref{Evolutionary-S1}}- {\eqref{Evolutionary-S2}}, together with the linearity
of $S^{-1}_{1}$ and $S^{-1}_{2}$, allow us to obtain that
\begin{equation*}
\begin{split} \|F({\mathbf{w}}_{1},w_{1})-F({\mathbf{w}}_{2},w_{2})&\|_{
\tilde{W}_{1}(0,T)\times W_{2}(0,T)}
\\
&\quad =\|S_{1}^{-1}(P_{1}({\mathbf{w}}_{1},w_{1})- P_{1}({\mathbf{w}}_{2},w_{2}))
\|_{\tilde{W}_{1}(0,T)}
\\
&\qquad \,\,+\|{S}_{2}^{-1}(P_{2}({\mathbf{w}}_{1},w_{1})- P_{2}({\mathbf{w}}_{2},w_{2}))
\|_{W_{2}(0,T)}
\\
&\quad \leq \,c_{1}(\mathrm{Re})\|P_{1}({\mathbf{w}}_{1},w_{1})-P_{1}({
\mathbf{w}}_{2},w_{2})\|_{L^{2}(0,T;L^{2}(\Omega )^{2})}
\\
&\qquad \,\,+C(1+\mathrm{Re}\,\mathrm{Pr})\|P_{2}({\mathbf{w}}_{1},w_{1})-P_{2}({
\mathbf{w}}_{2},w_{2})\|_{L^{2}(0, T;V_{2}')}.
\end{split}
\end{equation*}
From the identity
\begin{equation*}
\begin{split} &P_{1}({\mathbf{w}}_{1},w_{1})-P_{1}({\mathbf{w}}_{2},\tilde{w}_{2})=B_{1}({
\mathbf{w}}_{2},{\mathbf{w}}_{1}-{\mathbf{w}}_{2}) +B_{1}({\mathbf{w}}_{1}-{\mathbf{w}}_{2},{
\mathbf{w}}_{1})-A_{1}(0,w_{1}-w_{2}),
\end{split}
\end{equation*}
where $A_{1}(0,w_{1}-w_{2})(t):=A_{1}(0,w_{1}(t)-w_{2}(t))$, and estimates  {\eqref{A1}} and  {\eqref{Evolutionary-B1-Bochner}}, we have
\begin{equation*}
\begin{split} &\|P_{1}({\mathbf{w}}_{1},w_{1})-P_{1}({\mathbf{w}}_{2},w_{2})\|_{L^{2}(0,T;L^{2}(
\Omega )^{2})}
\\
&\quad \leq C(\|{\mathbf{w}}_{1}\|_{\tilde{W}_{1}(0,T)}+\|{\mathbf{w}}_{2}\|_{
\tilde{W}_{1}(0,T)})\|{\mathbf{w}}_{1}-{\mathbf{w}}_{2}\|_{\tilde{W}_{1}(0,T)}+C
\frac{\mathrm{Gr}}{\mathrm{Re}^{2}}\|w_{1}-w_{2}\|_{W_{2}(0,T)}
\\
&\quad \leq C\left (\tau +\frac{\mathrm{Gr}}{\mathrm{Re}^{2}}
\right )\|({\mathbf{w}}_{1},w_{1})-({\mathbf{w}}_{2},w_{2})\|_{\tilde{W}_{1}(0,T)
\times W_{2}(0,T)}.
\end{split}
\end{equation*}
Next, we estimate
$\|P_{2}({\mathbf{w}}_{1},w_{1})-P_{2}({\mathbf{w}}_{2},w_{2})\|_{L^{2}(0,T;V_{2}')}$.
In order to simplify notation, we define the following constants:
\begin{align*}
\ell _{u_{\mathrm{d}}}&:=\|u_{\mathrm{d}}\|_{W^{1/2,2}(\Gamma _{\mathrm{d}})},& \ell _{{\mathbf{v}}_{\mathrm{o}}}&:=\|{\mathbf{v}}_{\mathrm{o}}\|_{V_{1}},&
\ell _{u_{\mathrm{o}}}&:=\|u_{\mathrm{o}}\|_{L^{2}(\Omega )},
\\
\ell _{{\mathbf{g}}_{1}}&:=\|{\mathbf{g}}_{1}\|_{L^{2}(0,T;L^{2}(\Omega )^{2})},&
\ell _{g_{2}}&:=\|g_{2}\|_{L^{2}(0,T;V_{2}')},& \ell _{\beta }&:=\|
\beta \|_{L^{\infty }(\mathbb{R})}.
\end{align*}%
First, we consider
\begin{equation*}
P_{2}({\mathbf{w}}_{1},w_{1})-P_{2}({\mathbf{w}}_{2},w_{2})=\,I+J,
\end{equation*}
where
\begin{align*}
I:=\,&B_{2}({\mathbf{w}}_{2},w_{1}-w_{2})+B_{2}({\mathbf{w}}_{1}-{\mathbf{w}}_{2}, w_{1}+U),
\\
J:=\,& H({\mathbf{w}}_{1},w_{1}+U)-H({\mathbf{w}}_{2},w_{2}+U).
\end{align*}
Similarly as we estimated
$\|P_{1}({\mathbf{w}}_{1},w_{1})-P_{1}({\mathbf{w}}_{2},w_{2})\|_{L^{2}(0,T;L^{2}(
\Omega )^{2})}$, we find that
\begin{equation*}
\begin{split} &\|I\|_{L^{2}(0,T,V_{2}')}\leq C(\tau +T^{1/2}\ell _{u_{\mathrm{d}}})\|({\mathbf{w}}_{1},w_{1})-({\mathbf{w}}_{2},w_{2})\|_{\tilde{W}_{1}(0,T)
\times W_{2}(0,T)}.
\end{split}
\end{equation*}
In addition, we have (see the proof of  {Theorem~\ref{Stationary-Theorem1}})
\begin{equation*}
\begin{split} \|J(t)\|_{V_{2}'}\leq \,& C\ell _{\beta }\|w_{1}(t)-w_{2}(t)
\|_{V_{2}}\|{\mathbf{w}}_{1}(t)\|_{V_{1}}
\\
&+C\ell _{\beta }(\|w_{2}(t)\|_{V_{2}}+\ell _{u_{d}})\|{\mathbf{w}}_{1}(t)-{
\mathbf{w}}_{2}(t)\|_{V_{1}}
\\
&+CL(\|w_{2}(t)\|_{V_{2}}+\ell _{u_{d}})\|{\mathbf{w}}_{1}(t)-{\mathbf{w}}_{2}(t)
\|_{V_{1}}\|{\mathbf{w}}_{2}(t)\|_{V_{1}},
\end{split}
\end{equation*}
for almost every $t\in (0,T)$, and where $L>0$ is the Lipschitz constant
of $\beta $. From this, we find that
\begin{equation*}
\begin{split} &\|J\|_{L^{2}(0,T;V_{2}')}
\\
&\quad \leq C(L\tau ^{2}+(\ell _{\beta }+LT^{1/2}\ell _{u_{\mathrm{d}}})
\tau + T^{1/2}\,\ell _{u_{\mathrm{d}}}\,\ell _{\beta })\|({\mathbf{w}}_{1},w_{1})-({
\mathbf{w}}_{2},w_{2})\|_{\tilde{W}_{1}(0,T)\times W_{2}(0,T)}.
\end{split}
\end{equation*}
Therefore,
\begin{equation*}
\begin{split} &\|P_{2}({\mathbf{w}}_{1},w_{1})-P_{2}({\mathbf{w}}_{2},w_{2})\|_{L^{2}(0,T;V_{2}')}
\\
&\quad \leq C(L\tau ^{2}+(1+\ell _{\beta }+LT^{1/2}\ell _{u_{\mathrm{d}}})
\tau + T^{1/2}\,\ell _{u_{\mathrm{d}}}\,(1+\ell _{\beta }))
\\
&\qquad \times \|({\mathbf{w}}_{1},w_{1})-({\mathbf{w}}_{2},w_{2})\|_{\tilde{W}_{1}(0,T)
\times W_{2}(0,T)}.
\end{split}
\end{equation*}

Hence,
\begin{equation*}
\begin{split} &\|F({\mathbf{w}}_{1},w_{1})-F({\mathbf{w}}_{2},w_{2})\|_{
\tilde{W}_{1}(0,T)\times W_{2}(0,T)}\leq K_{1}\|({\mathbf{w}}_{1},w_{1})-({
\mathbf{w}}_{2},w_{2})\|_{\tilde{W}_{1}(0,T)\times W_{2}(0,T)},
\end{split}
\end{equation*}
where
\begin{equation*}
\begin{split} &K_{1}:=CL(1+\mathrm{Re}\,\mathrm{Pr})\tau ^{2}+ (c_{1}(
\mathrm{Re})+C(1+\mathrm{Re}\,\mathrm{Pr})(1+\ell _{\beta }+LT^{1/2}
\ell _{u_{\mathrm{d}}}))\tau
\\
&\qquad +c_{1}(\mathrm{Re})\frac{\mathrm{Gr}}{\mathrm{Re}^{2}}+ C
\,T^{1/2}\ell _{u_{\mathrm{d}}}(1+\mathrm{Re}\,\mathrm{Pr})(1+\ell _{\beta }).
\end{split}
\end{equation*}

From analogous arguments, we find that
\begin{equation*}
\begin{split} &\|F({\mathbf{w}},w)\|_{\tilde{W}_{1}(0,T)\times W_{2}(0,T)}
\leq K_{2},
\end{split}
\end{equation*}
for any $({\mathbf{w}},w)\in B_{\tau }$, where
\begin{equation*}
\begin{split} K_{2}:=\,& (c_{1}(\mathrm{Re})+C(1+\mathrm{Re}\,
\mathrm{Pr})(1+\ell _{\beta }))\tau ^{2}
\\
&+ \left (c_{1}(\mathrm{Re})\frac{\mathrm{Gr}}{\mathrm{Re}^{2}}+ CT^{1/2}
\ell _{u_{\mathrm{d}}}(1+\mathrm{Re}\,\mathrm{Pr})(1+\ell _{\beta })
\right )\tau
\\
&+c_{1}(\mathrm{Re})\ell _{{\mathbf{g}}_{1}}+T^{1/2}\ell _{u_{\mathrm{d}}}c_{1}(
\mathrm{Re})\frac{\mathrm{Gr}}{\mathrm{Re}^{2}}+c_{2}(\mathrm{Re})
\ell _{{\mathbf{v}}_{\mathrm{o}}}
\\
&+C\left (T^{1/2}(1+\mathrm{Re}^{-1}\,\mathrm{Pr}^{-1})\ell _{u_{\mathrm{d}}}+(1+\mathrm{Re}\,\mathrm{Pr})(\ell _{g_{2}}+\ell _{u_{\mathrm{o}}}+\ell _{u_{\mathrm{d}}})\right ).
\end{split}
\end{equation*}

Looking at the definitions of $K_{1}$ and $K_{2}$, we notice that there
exist $\tau >0$ and
$\varepsilon _{1},\varepsilon _{2},\varepsilon _{3}>0$ such that if  {\eqref{A5}} holds true, then
\begin{equation*}
K_{1}<1\qquad \text{and}\qquad K_{2}\leq \tau ,
\end{equation*}
provided that $\ell _{{\mathbf{g}}_{1}}$, $\ell _{g_{2}}$,
$\ell _{u_{\mathrm{d}}}$, $\ell _{{\mathbf{v}}_{\mathrm{o}}}$, and
$\ell _{u_{\mathrm{o}}}$ are sufficiently small. Therefore, $F$ is a contraction
from $B_{\tau }$ into itself.
\end{proof}

\section{Numerical results}\label{Section:numerics}
This section is devoted to analyzing the performance of the artificial boundary
condition proposed in this paper:
%
\begin{equation}
\label{HT-bis}
\qquad \frac{1}{\mathrm{Re}\,\mathrm{Pr}}\,
\frac{\partial u}{\partial {\mathbf{n}}}- u\,\beta ({\mathbf{v}}\cdot {\mathbf{n}})\,({
\mathbf{v}}\cdot {\mathbf{n}})=0\qquad \text{on}\quad \Gamma _{\mathrm{o}},
\end{equation}
where $\Gamma _{\mathrm{o}}$ represents an open/artificial boundary of a
truncated domain $\Omega $, and $\beta $ satisfies $(\mathrm{A1})$. We
perform a variety of tests to compare solutions on $\Omega $ using  {\eqref{HT-bis}} with respect to the restriction to $\Omega $ of a reference solution
to a problem on a larger domain $\Omega ^{\mathrm{ext}}\supset\Omega $. In all our tests we consider $\mathbf{g}_{1}\equiv 0$ and
$g_{2}\equiv 0$ in order to fully concentrate on buoyancy effects.

\subsection*{Geometrical setup and boundary conditions}
We consider a geometry that represents a 2-dimensional open cavity, similar
to that considered by Chan and Tien in \cite{ChTi1985}. Let
\begin{equation*}
\Omega :=(0,1)^{2},\quad
\Gamma _{\mathrm{o}}:=\{(1,x_{2}): x_{2}\in (0,1)\},\quad\text{and}\quad
\Omega ^{\mathrm{ext}}:=\Omega
\cup \Gamma_{\mathrm{o}}\cup ((1,2)\times (-1,2)).
\end{equation*}
The domain $\Omega $ represents the cavity and, in what follows, we refer to it as the \emph{truncation} of $\Omega ^{\mathrm{ext}}$ at the line $\Gamma_{\mathrm{o}}$, which is the open boundary for $\Omega $; see  {Fig.~\ref{Geometry}}. Further,
we refer to $\Omega ^{\mathrm{ext}}$ as the \emph{extension} of
$\Omega $ and define
$\Gamma ^{\mathrm{ext}}:=\partial \Omega ^{\mathrm{ext}}$.

\begin{figure}[ht]
\vspace{.5cm}
\centering
		\includegraphics[scale=.3]{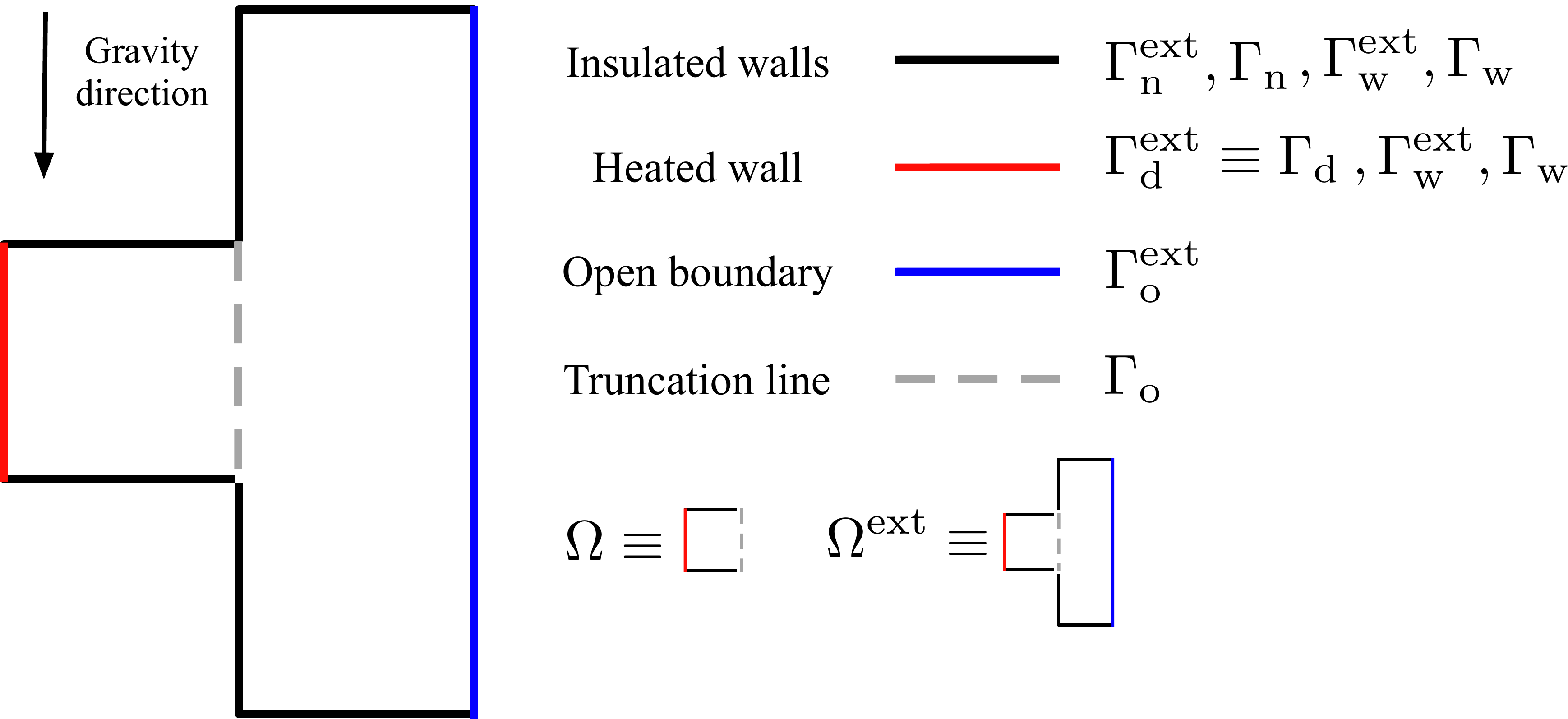}	
	\caption{\small Geometry of the setup for $\Omega$ and $\Omega^{\mathrm{ext}}$.}\label{Geometry}
\end{figure}

On the extended domain $\Omega ^{\mathrm{ext}}$, we consider the following
boundary decomposition and conditions:
\begin{align*}
{\mathbf{v}}&= 0&\text{on}&\quad \Gamma _{\mathrm{w}}^{\mathrm{ext}},& u&=1&\text{on}&\quad
\Gamma _{\mathrm{d}}^{\mathrm{ext}},&%
\\
\frac{1}{\mathrm{Re}}\,\frac{\partial {\mathbf{v}}}{\partial {\mathbf{n}}}&=p\,{
\mathbf{n}}&\text{on}&\quad \Gamma ^{\mathrm{ext}}_{\mathrm{o}},&
\frac{\partial u}{\partial {\mathbf{n}}}&= 0&\text{on}&\quad \Gamma _{\mathrm{n}}^{\mathrm{ext}} \cup \Gamma ^{\mathrm{ext}}_{\mathrm{o}},&
\end{align*}
where
\begin{align*}
\Gamma ^{\mathrm{ext}}_{\mathrm{d}}&:=\{(0,x_{2}) : x_{2}\in (0,1)\},&\qquad
\Gamma ^{\mathrm{ext}}_{\mathrm{o}}&:=\{(2,x_{2}) : x_{2}\in (-1,2)\},
\\
\Gamma _{\mathrm{n}}^{\mathrm{ext}}&:= \Gamma ^{\mathrm{ext}}
\setminus (
\overline{\Gamma ^{\mathrm{ext}}_{\mathrm{o}}\cup \Gamma _{\mathrm{d}}^{\mathrm{ext}}}),&\qquad
\Gamma _{\mathrm{w}}^{\mathrm{ext}}&:= \Gamma ^{\mathrm{ext}}
\setminus (\overline{\Gamma ^{\mathrm{ext}}_{\mathrm{o}}}).
\end{align*}
Specifically, the above conditions determine that we consider a heating
element on the left vertical wall
$\Gamma _{\mathrm{d}}^{\mathrm{ext}}$, and the rest of the boundary of
$\Omega ^{\mathrm{ext}}$ is \emph{insulated}. Furthermore, flow is allowed
to leave the extended domain on the right vertical part of the boundary,
$\Gamma ^{\mathrm{ext}}_{\mathrm{o}}$, and in the rest of the boundary a
no-slip condition is imposed. See  {Fig.~\ref{Geometry}} for clarification.
The solution on this domain is called \emph{reference solution}.

For the domain of interest, $\Omega \subset \Omega ^{\mathrm{ext}}$, we
consider:
\begin{align*}
u&=1 \quad \text{on}\quad \Gamma _{\mathrm{d}},\qquad \qquad
\frac{\partial u}{\partial {\mathbf{n}}}= 0\quad \text{on}\quad \Gamma _{\mathrm{n}},\qquad \qquad {\mathbf{v}}=0 \quad \text{on} \quad \Gamma _{
\mathrm{w}},
\end{align*}
where
\begin{align*}
\Gamma _{\mathrm{w}}&:=\Gamma \setminus
\overline{\Gamma _{\mathrm{o}}},\qquad \Gamma _{\mathrm{d}}:=\{(0, x_{2}):x_{2}
\in (0,1)\},
\\
\Gamma _{\mathrm{n}}&:= \{(x_{1},x_{2}) : x_{1}\in (0,1), x_{2}=0
\text{ or } x_{2}=1\}.
\end{align*}
See  {Fig.~\ref{Geometry}} for the graphic description of boundary parts.
In order to show the performance of the boundary condition proposed in
this paper, we consider several choices of artificial conditions on
$\Gamma _{\mathrm{o}}$. We assume that the fluid flow satisfies either
a do-nothing or a ``directional'' do-nothing condition, which consists in
adding the extra term
$-\frac{1}{2}\,{\mathbf{v}}\,({\mathbf{v}}\cdot {\mathbf{n}})_{-}$ on the left-hand side
of the do-nothing condition. This modification was recently introduced
for isothermal fluids in \cite{BrMu2014}, where the authors obtain accurate
results and enhance stability properties of the system with respect to
the standard do-nothing condition. We denote do-nothing and directional
do-nothing conditions by $(\mathrm{DN})$ and $(\mathrm{DDN})$, respectively,
that is:
\begin{align*}
(\mathrm{DN})\quad \frac{1}{\mathrm{Re}}\,
\frac{\partial {\mathbf{v}}}{\partial {\mathbf{n}}}&=p\,{\mathbf{n}}, \qquad \qquad (
\mathrm{DDN}) \quad \frac{1}{\mathrm{Re}}\,
\frac{\partial {\mathbf{v}}}{\partial {\mathbf{n}}}-\frac{1}{2}\,{\mathbf{v}}\,({\mathbf{v}}
\cdot {\mathbf{n}})_{-}=p\,{\mathbf{n}}.
\end{align*}
In addition, we assume that the artificial boundary condition  {\eqref{HT-bis}} holds true at the line of truncation. The case when
$\beta \equiv 0$ corresponds to a homogeneous Neumann condition. Then,
we denote condition  {\eqref{HT-bis}} differently according to
$\beta \equiv 0$ or $\beta \not\equiv 0$, as follows:
\begin{align*}
(\mathrm{N})\quad \frac{\partial u}{\partial {\mathbf{n}}}=0, \qquad
\qquad (\mathrm{N}_{\beta })\quad
\frac{1}{\mathrm{Re}\,\mathrm{Pr}}\,
\frac{\partial u}{\partial {\mathbf{n}}}= u\,\beta ({\mathbf{v}}\cdot {\mathbf{n}})\,({
\mathbf{v}}\cdot {\mathbf{n}}).
\end{align*}
In particular, we consider the following two choices for the function
$\beta $:
\begin{align*}
\beta _{1}(s) = 1/2 - 1/\pi \cdot \arctan (100s), \qquad \qquad
\beta _{2}(s)=
\begin{cases}
1/2 &\mbox{if } s <0,
\\
0 & \mbox{otherwise},
\end{cases}
\end{align*}
and observe that $\beta _{2}$ corresponds to the boundary condition considered
in \cite{PeThBlCr2008-a,PeThBlCr2008-b} as an ad-hoc modification on the
homogeneous Neumann condition for analysis purposes.

\subsubsection*{Discretization and solver details} 
For discretization, we follow Elman et al. \cite{MR2783823}. We subdivide
the time interval $[0,T]$ into $N$ sub-intervals of length $k$ and semi-discretize
in time  {\eqref{BoussinesqEvolutionary}},
 {\eqref{BoussinesqEvolutionary2}}, and  {\eqref{BoussinesqEvolutionary3}} in
Problem $(\tilde{\mathbb{P}})$ by applying the trapezoid rule. In this
manner, we obtain that
%
\begin{align}
\label{eq:semi_discrete_Be}
\frac{2}{k}{\mathbf{v}}^{n+1} + {\mathbf{v}}^{n+1}\cdot \nabla {\mathbf{v}}^{n+1} -
\frac{1}{\mathrm{Re}}\Delta {\mathbf{v}}^{n+1} + \nabla p^{n+1} -
\frac{\mathrm{Gr}}{\mathrm{Re}^{2}}u{\mathbf{e}} & = \frac{2}{k}{\mathbf{v}}^{n}+D_{t}
\mathbf{v}^{n},
\\
\label{eq:divfree}
\mathrm{div}\,{\mathbf{v}}^{n+1} & = 0,
\\
\label{eq:temptrap}
\frac{2}{k}u^{n+1} + {\mathbf{v}}^{n+1}\cdot \nabla u^{n+1} -
\frac{1}{\mathrm{Re}\,\mathrm{Pr}}\Delta u^{n+1} & = \frac{2}{k}u^{n}+D_{t}u^{n},
\end{align}
for $n=0,1,\ldots , N-1$, where
\begin{align*}
&D_{t}\mathbf{v}^{n}:=-{\mathbf{v}}^{n}\cdot \nabla {\mathbf{v}}^{n} +
\frac{1}{\mathrm{Re}}\Delta {\mathbf{v}}^{n} - \nabla p^{n} +
\frac{\mathrm{Gr}}{\mathrm{Re}^{2}}u^{n}\mathbf{e},
\\
&D_{t}u^{n}:=-\; {\mathbf{v}}^{n}\cdot \nabla u^{n}+
\frac{1}{\mathrm{Re}\,\mathrm{Pr}}\Delta u^{n},
\end{align*}
and ${\mathbf{v}}^{0}, u^{0}, p^{0}$ are assumed to be known. We notice that
boundary conditions in problem $(\tilde{\mathbb{P}})$ remain the same for
the semi-discretized problem since they do not depend on the time variable.

Then, we consider a linearization of a standard weak form of  {(\ref{eq:semi_discrete_Be})}- {(\ref{eq:temptrap})}
on the base of approximating ${\mathbf{v}}^{n+1}$ and $u^{n+1}$ by the linear
extrapolations
$\tilde{{\mathbf{v}}}^{n+1} := 2\,{\mathbf{v}}^{n}-{\mathbf{v}}^{n-1}$ and
$\tilde{u}^{n+1} := 2u^{n}-u^{n-1}$. Thus, for a test triplet
$({\mathbf{y}}, q, y)$ the linearized weak formulation of  {\eqref{eq:semi_discrete_Be}} is given by
%
\begin{equation}
\begin{split}
\label{weakkk1}
&\frac{2}{k}\int _{\Omega }{\mathbf{v}}^{n+1}\cdot {\mathbf{y}}\,\mathrm{d}x + \int _{\Omega }(\tilde{{\mathbf{v}}}^{n+1}\cdot \nabla {\mathbf{v}}^{n+1})\cdot {\mathbf{y}}
\,\mathrm{d}x + \frac{1}{\mathrm{Re}}\int _{\Omega }\nabla {\mathbf{v}}^{n+1}
\cdot \nabla {\mathbf{y}}\,\mathrm{d}x
\\
& \quad - \int _{\Omega }p^{n+1}\cdot \nabla {\mathbf{y}}\,\mathrm{d}x -
\frac{\mathrm{Gr}}{\mathrm{Re}^{2}}\int _{\Omega }u^{n+1} {\mathbf{e}}
\cdot {\mathbf{y}}\,\mathrm{d}x - C_{\mathbf{v}}({\mathbf{v}}^{n+1},
\tilde{{\mathbf{v}}}^{n+1}, {\mathbf{y}})
\\
&\qquad = \frac{2}{k}\int _{\Omega }{\mathbf{v}}^{n}\cdot {\mathbf{y}}\,\mathrm{d}x -
\int _{\Omega }(\tilde{{\mathbf{v}}}^{n+1}\cdot \nabla {\mathbf{v}}^{n+1})\cdot {
\mathbf{y}}\,\mathrm{d}x - \frac{1}{\mathrm{Re}}\int _{\Omega }\nabla {\mathbf{v}}^{n}
\cdot \nabla {\mathbf{y}}\,\mathrm{d}x
\\
& \qquad \quad + \int _{\Omega }p^{n+1}\cdot \nabla {\mathbf{y}}\,
\mathrm{d}x + \frac{\mathrm{Gr}}{\mathrm{Re}^{2}} \int _{\Omega }u^{n}
{\mathbf{e}}\cdot {\mathbf{y}}\,\mathrm{d}x + C_{\mathbf{v}}({\mathbf{v}}^{n}, {\mathbf{v}}^{n},
{\mathbf{y}}),
\end{split}
\end{equation}
where $C_{\mathbf{v}}$ depends on the choice of the boundary condition: For
the directional do-nothing condition $(\mathrm{DDN})$, the linearized boundary
term is
\begin{equation*}
C_{\mathbf{v}}({\mathbf{v}}_{1}, {\mathbf{v}}_{2}, {\mathbf{y}}) = \int _{\Gamma _{\mathrm{o}}}\frac{1}{2}({\mathbf{v}}_{1}\cdot {\mathbf{y}})({\mathbf{v}}_{2}\cdot {\mathbf{n}})_{-}
\,\mathrm{d}\sigma ,
\end{equation*}
and for the do-nothing condition $(\mathrm{DN})$, is
$C_{\mathbf{v}} \equiv 0$.

The weak formulation of the divergence free condition  {\eqref{eq:divfree}} is given by
%
\begin{align}
\int _{\Omega }\mathrm{div}\,{\mathbf{v}}\cdot q\,\mathrm{d}x=0,
\label{eq:linear_system_div_eq}
\end{align}
and the one for  {\eqref{eq:temptrap}} is obtained as
%
\begin{equation}
\label{eq:temptrap2}
\begin{split} &\frac{2}{k}\int _{\Omega }u^{n+1}\cdot y\,\mathrm{d}x +
\int _{\Omega }(\tilde{{\mathbf{v}}}^{n+1}\cdot \nabla u^{n+1})\cdot y\,
\mathrm{d}x+ \frac{1}{\mathrm{Re}\mathrm{Pr}}\int _{\Omega }\nabla u^{n+1}
\nabla y\,\mathrm{d}x
\\
&\quad - C_{u}(u^{n+1}, \tilde{{\mathbf{v}}}^{n+1}, y)
\\
&\qquad = \frac{2}{k}\int _{\Omega }u^{n}\cdot y\,\mathrm{d}x - \int _{\Omega }(\tilde{{\mathbf{v}}}^{n}\cdot \nabla u^{n})\cdot y\,\mathrm{d}x -
\frac{1}{\mathrm{Re}\mathrm{Pr}}\int _{\Omega }\nabla u^{n}\nabla y\,
\mathrm{d}x + C_{u}(u^{n}, {\mathbf{v}}^{n}, y),
\end{split}
\end{equation}
where for $(\mathrm{N}_{\beta })$, we have
\begin{equation*}
C_{u}({\mathbf{v}}, u, y) = \int _{\Gamma _{\mathrm{o}}}u\beta _{i}({\mathbf{v}}
\cdot {\mathbf{n}})({\mathbf{v}}\cdot {\mathbf{n}}) y\,\mathrm{d}\sigma ,
\end{equation*}
and in case of the homogeneous Neumann condition $(\mathrm{N})$, we have
$C_{u} \equiv 0$.

Discretization in space is done by the application of finite elements method
on a regular triangular mesh. The discrete approximations for velocity
and pressure are computed by piecewise quadratic, respectively linear,
functions on mixed $P_{2}$-$P_{1}$ Taylor-Hood triangular elements, and
temperature is approximated by piecewise quadratic functions on
$P_{2}$ elements as well. The variational form is directly obtained from
 {(\ref{weakkk1})}- {(\ref{eq:temptrap2})} by replacing the functions by their
respective finite-dimensional approximations.

The discrete problem is then solved with FEniCS, under the use of LU-decomposition
for solving the linear systems. For each time-step, we utilize that
${\mathbf{v}}^{n+1}$ does not occur in the linearization of the temperature
equation, which allows  {\eqref{eq:temptrap2}} to be solved separately.

\subsection*{Test Results}

Computations are produced for
\begin{equation*}
\mathrm{Re}\in \{2,3,4,5\}, \quad \mathrm{Gr}\in \{500, 1000, 2000
\}, \quad \text{and} \quad \mathrm{Pr}= 1,
\end{equation*}
in the time interval $[0,T]$ for $T=1$.

In order to provide a qualitative idea of solutions, the reference velocity
field and temperature distribution on $\Omega ^{\mathrm{ext}}$ are depicted
in  {Fig.~\ref{fig:vu_r0_side_by_side}} for $\mathrm{Re}=3$,
$\mathrm{Gr}=10^{3}$ and final time $t=T=1$. Furthermore, a solution on
$\Omega $ is shown in  {Fig.~\ref{fig:vu_c3_DN}} for boundary conditions
$(\mathrm{DN})-(\mathrm{N_{\beta _{1}}})$. All numerical solutions are
computed on a regular triangulation with 6489 nodes for a time-discretization
parameter of $k=10^{-2}$.

\begin{figure}[ht]
	\centering
	\begin{subfigure}{0.49\textwidth}	
\includegraphics[width=\textwidth]{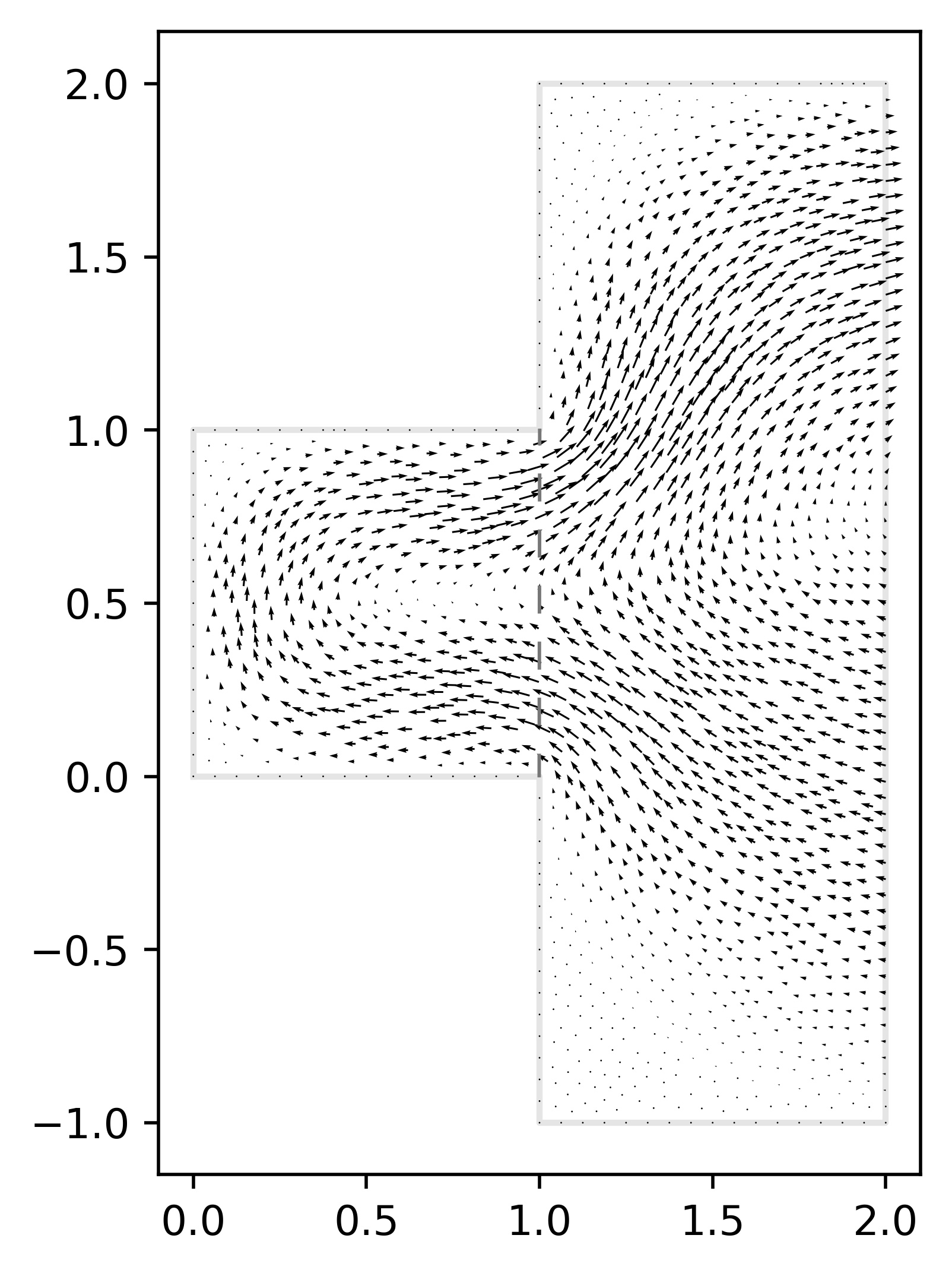}
	\end{subfigure}
	\begin{subfigure}{0.49\textwidth}
		\includegraphics[width=\textwidth]{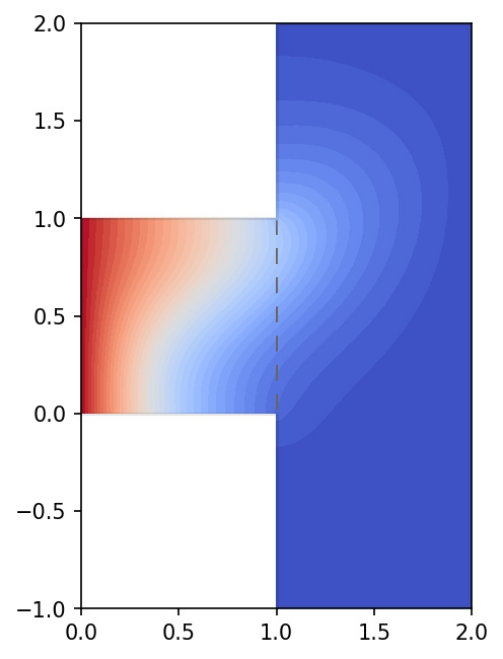}
	\end{subfigure}
	\caption{\small Velocity and temperature of the reference solution on $\Omega^{\mathrm{ext}}$ at $t=T=1$ for $\Rey=3.0$ and $\Gra=10^3$. In the heatmap plot (right), {\color{red}red} corresponds to $u=1$ and {\color{blue}blue} corresponds to $u=0$.}\label{fig:vu_r0_side_by_side}
\end{figure}

\begin{figure}[ht]
	\centering
	\begin{subfigure}{0.48\textwidth}
		\includegraphics[width=\textwidth]{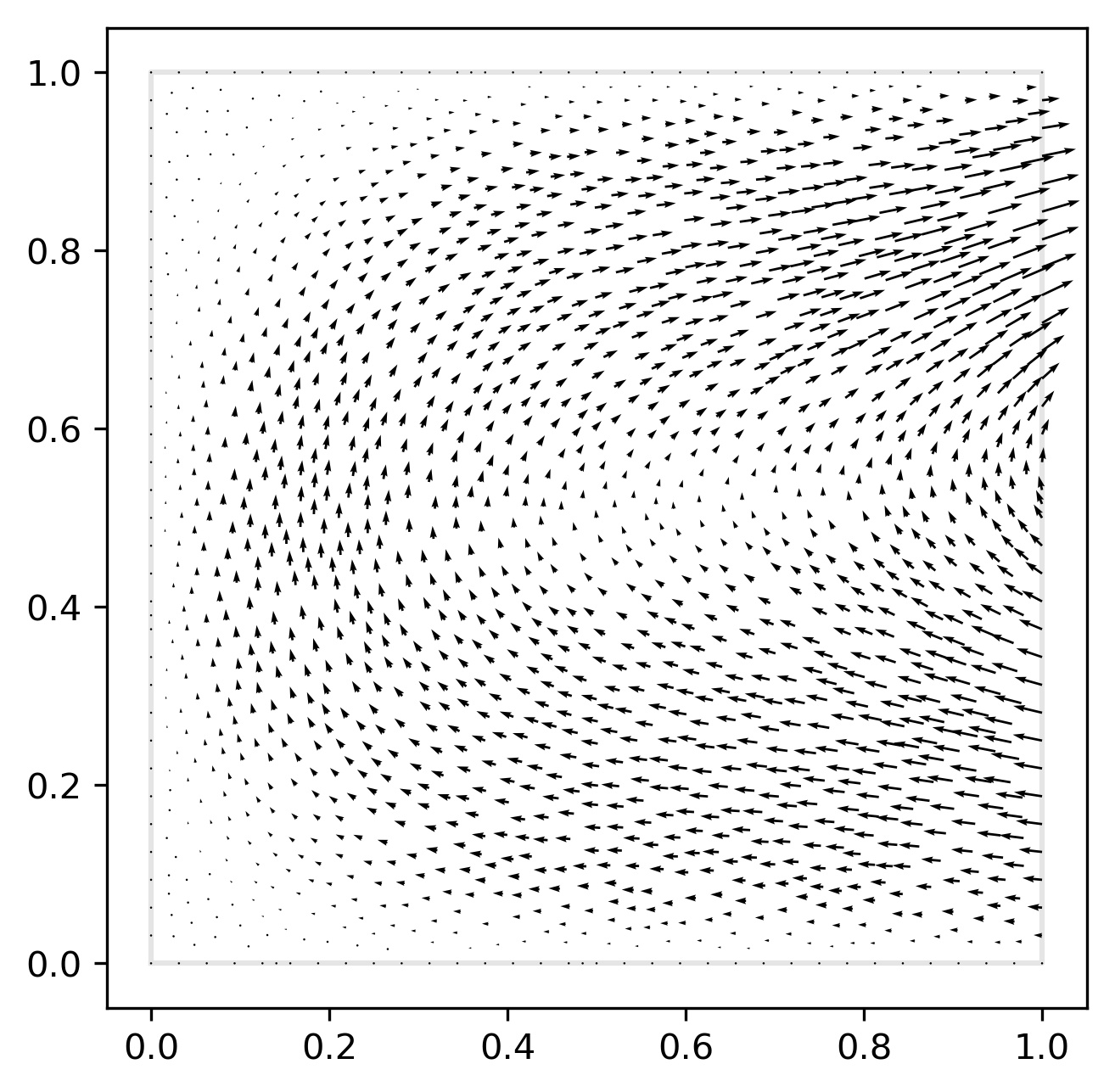}
	\end{subfigure}
	\hspace{0.8em}
	\begin{subfigure}{0.48\textwidth}
		\includegraphics[width=\textwidth]{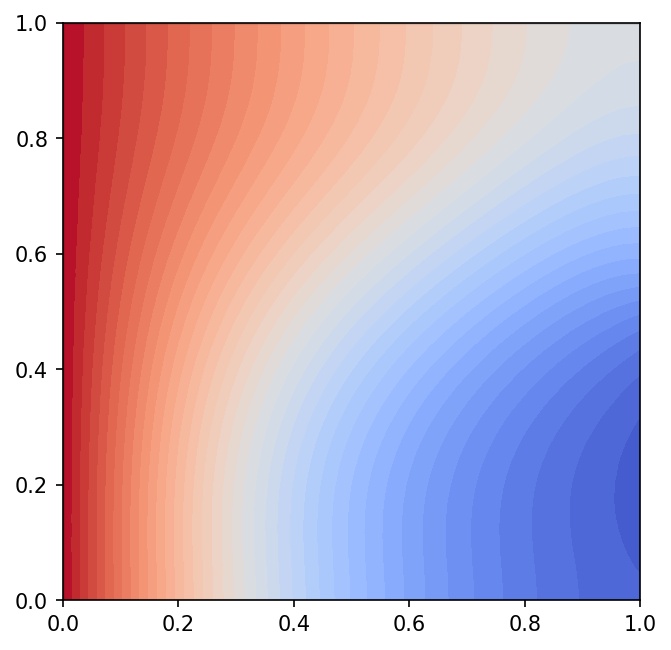}
	\end{subfigure}
	\caption{\small Velocity and temperature of the solution on the truncated domain $\Omega$ at $t=T=1$ for the artificial boundary conditions $(\mathrm{DN})-(\mathrm{N_{\beta_1}})$ for $\Rey=3.0$ and $\Gra=10^3$.}\label{fig:vu_c3_DN}
\end{figure}

With the aim to measure the performance of the solutions
$(\mathbf{v}^{n},u^{n})$ for several combinations of boundary conditions,
we compare them at the final time $t=T=1$ or $n=N$ to the reference solution
$(\mathbf{v}_{\mathrm{ref}}^{n},u_{\mathrm{ref}}^{n})$ by
%
\begin{equation}
\label{eq:res_square_def}
\mathrm{res}_{\Omega }= \|\nabla {\mathbf{v}}^{N} -\nabla {\mathbf{v}}^{N}_{
\text{ref}}\|^{2}_{2} + \|\nabla u^{N} -\nabla u^{N}_{\text{ref}}\|^{2}_{2},
\end{equation}
where $\|\cdot \|_{2}$ denotes the norm in either
$L^{2}(\Omega )^{2\times 2}$ or $L^{2}(\Omega )^{2}$. The values are displayed
in  {Table~\ref{fig:table_residuals_on_square}}. Additionally, we define residuals
over the artificial boundary in the trace sense by
%
\begin{equation}
\label{eq:res_def}
\mathrm{res}_{\Gamma _{\mathrm{o}}} = \|{\mathbf{v}}^{N} -{\mathbf{v}}^{N}_{
\text{ref}}\|^{2}_{L^{2}({\Gamma _{\mathrm{o}}})^{2}} + \|u^{N} -u^{N}_{
\text{ref}}\|^{2}_{L^{2}({\Gamma _{\mathrm{o}}})},
\end{equation}
for which the values are given in  {Table~\ref{fig:table_profile_difference_tables}}. It is clearly seen that both
indicators, $\mathrm{res}_{\Omega }$ and
$\mathrm{res}_{\Gamma _{\mathrm{o}}}$, are consistently improved by the
use of the boundary condition proposed in this paper.

For $\mathrm{Re}=3$ and $\mathrm{Gr}=10^{3}$, the profiles of temperature
and the horizontal component of velocity field are shown in  {Fig.~\ref{fig:profiles_default_comparison}}. One can observe the positive effect
of the proposed boundary condition almost on every point of
$\Gamma _{\mathrm{o}}$. For the velocity profile the advantage of the directional
do-nothing over the do-nothing condition is significant as well.

\newcommand{\vhorizontal}{$|v_{horizontal}|$}
\newcommand{\vfunc}{$|v|$}
\newcommand{\ufunc}{$|u|$}
\newcommand{\vandu}{$|v|+|u|$}
\definecolor{aplot}{RGB}{40,117,169}
\definecolor{bplot}{RGB}{243,129,30}
\definecolor{cplot}{RGB}{51, 153,51}
\definecolor{dplot}{RGB}{198,44 ,44}
\definecolor{eplot}{RGB}{147,105,179}
\definecolor{fplot}{RGB}{136,88 ,78}
\definecolor{gplot}{RGB}{217,125,188}
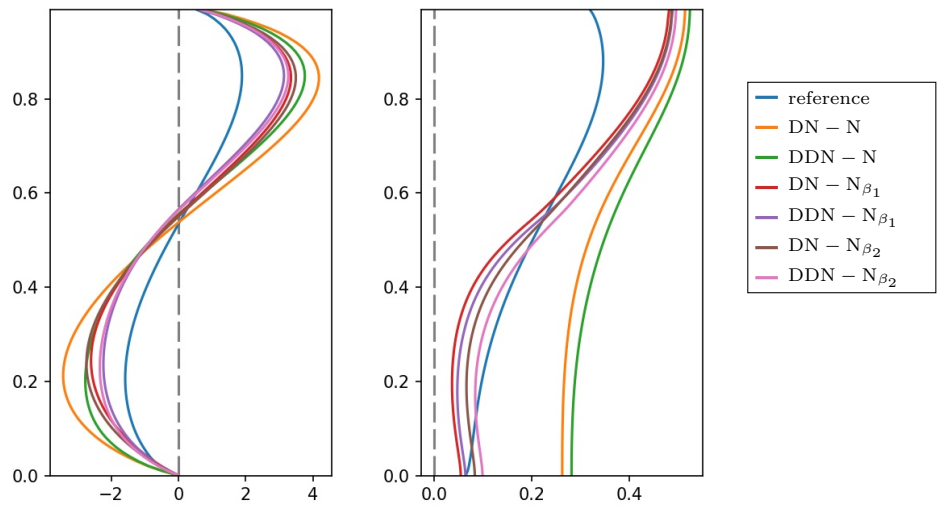
\begin{figure}[ht]
\begin{subfigure}{0.37\textwidth}
\includegraphics[width=\textwidth]{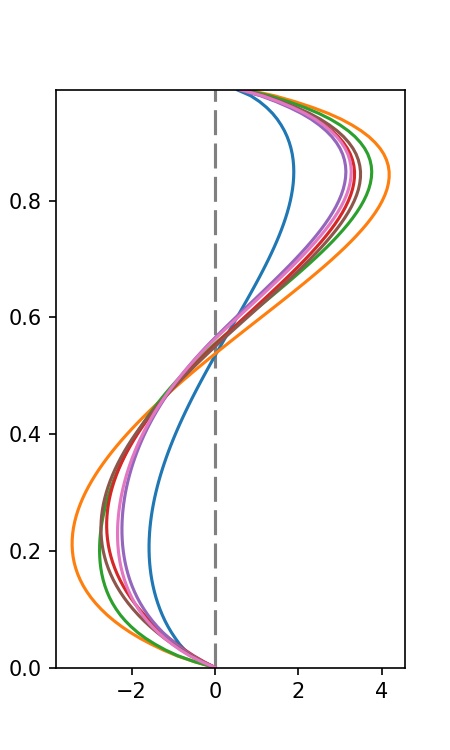}
\caption{\small Normal velocity $\mathbf{v}
\cdot\mathbf{n}$ at $x_1=1$ and time $t=T=1$.}
\end{subfigure}
\begin{subfigure}{0.37\textwidth}
\includegraphics[width=\textwidth]{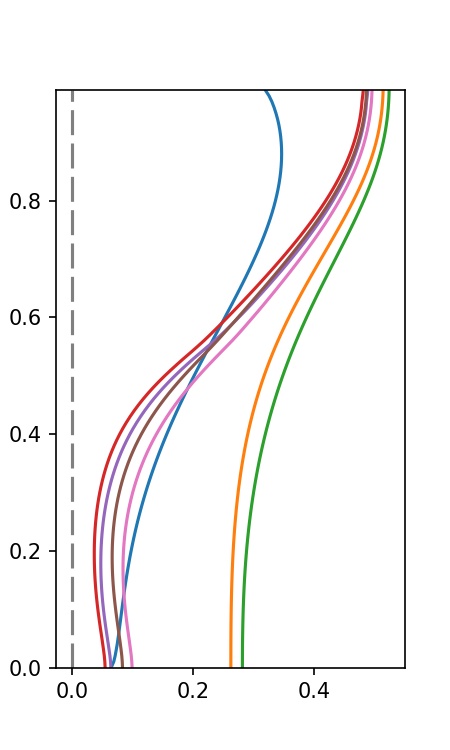}
\caption{\small Temperature $u$ at $x_1=1$ and time $t=T=1$.}
\end{subfigure}
		\begin{tikzpicture}
			\draw[thin] (0,0) rectangle (2.5, 2.8);
			\draw[very thick, color=aplot] (0.1, 2.6) -- (0.4, 2.6) node[anchor=west, color=black]  {\scriptsize{reference}};
			\draw[very thick, color=bplot] (0.1, 2.2) -- (0.4, 2.2) node[anchor=west, color=black]  {\scriptsize{$\mathrm{DN}-\mathrm{N}$}};
			\draw[very thick, color=cplot] (0.1, 1.8) -- (0.4, 1.8) node[anchor=west, color=black] {\scriptsize{$\mathrm{DDN}-\mathrm{N}$}};
			\draw[very thick, color=dplot] (0.1, 1.4) -- (0.4, 1.4) node[anchor=west, color=black] {\scriptsize{$\mathrm{DN}-\mathrm{N_{\beta_1}}$}};
			\draw[very thick, color=eplot] (0.1, 1.0) -- (0.4, 1.0) node[anchor=west, color=black]   {\scriptsize{$\mathrm{DDN}-\mathrm{N_{\beta_1}}$}};
			\draw[very thick, color=fplot] (0.1, 0.6) -- (0.4, 0.6) node[anchor=west, color=black] {\scriptsize{$\mathrm{DN}-\mathrm{N_{\beta_2}}$}};
			\draw[very thick, color=gplot] (0.1, 0.2) -- (0.4, 0.2) node[anchor=west, color=black]   {\scriptsize{$\mathrm{DDN}-\mathrm{N_{\beta_2}}$}};
		\end{tikzpicture}
	\caption{\small Comparison of velocity and temperature profiles for $x_1=1$ at $t=T=1$ for $\Rey=3.0$ and $\Gra=10^3$.}
	\label{fig:profiles_default_comparison}
\end{figure}

\begin{landscape}
\begin{table}[ht]
\begin{subtable}[t]{\linewidth}
\centering
\begin{tabular}{lllllll}
                       & $\mathrm{DN}-\mathrm{N}$         & $\mathrm{DDN}-\mathrm{N}$        & $\mathrm{DN}-\mathrm{N_{\beta_1}}$    & $\mathrm{DDN}-\mathrm{N}_{\beta_1}$             & $\mathrm{DN}-\mathrm{N}_{\beta_2}$    & $\mathrm{DDN}-\mathrm{N}_{\beta_2}$    \\ \hline\\[-0.9em]
$\Rey = 2, \Gra = 500 $&$4.5686\se{1}$&$3.0640\se{1}$&$1.2740\se{1}$&$\mathbf{1.0262\se{1} }$&$1.6526\se{1}$&$1.3113\se{1} $\\
$\Rey = 2, \Gra = 1000$&$9.2430\se{1}$&$5.9206\se{1}$&$2.1157\se{1}$&$\mathbf{1.6392\se{1} }$&$2.4998\se{1}$&$1.9322\se{1} $\\
$\Rey = 2, \Gra = 2000$&$1.3692\se{2}$&$9.3841\se{1}$&$3.6147\se{1}$&$\mathbf{2.6388\se{1} }$&$3.9475\se{1}$&$2.8885\se{1} $\\
$\Rey = 3, \Gra = 500 $&$1.1321\se{1}$&$8.0385      $&$5.6257      $&$\mathbf{4.3795       }$&$6.9018      $&$5.2771       $\\
$\Rey = 3, \Gra = 1000$&$2.5821\se{1}$&$1.6485\se{1}$&$1.0132\se{1}$&$\mathbf{7.5600       }$&$1.1918\se{1}$&$8.7941       $\\
$\Rey = 3, \Gra = 2000$&$4.1072\se{1}$&$2.7252\se{1}$&$1.6176\se{1}$&$\mathbf{1.2097\se{1} }$&$1.7552\se{1}$&$1.3078\se{1} $\\
$\Rey = 4, \Gra = 500 $&$3.6454      $&$2.7909      $&$2.5294      $&$\mathbf{2.0110       }$&$2.8790      $&$2.2633       $\\
$\Rey = 4, \Gra = 1000$&$1.0767\se{1}$&$7.0715      $&$6.1112      $&$\mathbf{4.3871       }$&$7.0493      $&$4.9907       $\\
$\Rey = 4, \Gra = 2000$&$1.7968\se{1}$&$1.1480\se{1}$&$9.5475      $&$\mathbf{6.9977       }$&$1.0305\se{1}$&$7.4671       $\\
$\Rey = 5, \Gra = 500 $&$1.3164      $&$1.0817      $&$1.0841      $&$\mathbf{9.0367\se{-1}}$&$1.1688      $&$9.6915\se{-1}$\\
$\Rey = 5, \Gra = 1000$&$4.8294      $&$3.4058      $&$3.4933      $&$\mathbf{2.5726       }$&$3.8696      $&$2.8191       $\\
$\Rey = 5, \Gra = 2000$&$1.1505\se{1}$&$7.0829      $&$7.2493      $&$\mathbf{4.9087       }$&$7.9045      $&$5.2928       $   
\end{tabular}
\caption{\small Comparison of $\mathrm{res}_{\Omega}^{N}$ for solutions on the truncated domain $\Omega$ with the regarded open boundary conditions.}\label{fig:table_residuals_on_square}
\end{subtable}

\begin{subtable}[t]{\linewidth}
\centering
\begin{tabular}{lllllll}
& $\mathrm{DN}-\mathrm{N}$ & $\mathrm{DDN}-\mathrm{N}$ & $\mathrm{DN}-\mathrm{N}_{\beta_1}$ & $\mathrm{DDN}-\mathrm{N}_{\beta_1}$ & $\mathrm{DN}-\mathrm{N}_{\beta_2}$ & $\mathrm{DDN}-\mathrm{N}_{\beta_2}$ \\ \hline\\[-0.9em]
$\Rey = 2, \Gra = 500 $&$7.7872       $&$4.8659       $&$2.0599       $&$\mathbf{1.4807       }$&$2.7955       $&$2.0111       $\\
$\Rey = 2, \Gra = 1000$&$1.5142\se{1} $&$8.4914       $&$2.8882       $&$\mathbf{1.7664       }$&$3.7089       $&$2.3541       $\\
$\Rey = 2, \Gra = 2000$&$2.1123\se{1} $&$1.2059\se{1} $&$4.2178       $&$\mathbf{2.1686       }$&$4.9436       $&$2.6850       $\\
$\Rey = 3, \Gra = 500 $&$1.9878       $&$1.3626       $&$9.9941\se{-1}$&$\mathbf{7.3864\se{-1}}$&$1.2320       $&$8.9852\se{-1}$\\
$\Rey = 3, \Gra = 1000$&$4.3320       $&$2.4773       $&$1.5613       $&$\mathbf{9.9303\se{-1}}$&$1.9232       $&$1.2313       $\\
$\Rey = 3, \Gra = 2000$&$6.2717       $&$3.2784       $&$1.8124       $&$\mathbf{9.0284\se{-1}}$&$2.1294       $&$1.1164       $\\
$\Rey = 4, \Gra = 500 $&$6.5636\se{-1}$&$4.9420\se{-1}$&$4.6102\se{-1}$&$\mathbf{3.5705\se{-1}}$&$5.2341\se{-1}$&$4.0193\se{-1}$\\
$\Rey = 4, \Gra = 1000$&$1.8422       $&$1.1366       $&$1.0493       $&$\mathbf{6.9750\se{-1}}$&$1.2213       $&$8.0378\se{-1}$\\
$\Rey = 4, \Gra = 2000$&$2.7347       $&$1.3825       $&$1.1693       $&$\mathbf{6.2483\se{-1}}$&$1.3437       $&$7.3074\se{-1}$\\
$\Rey = 5, \Gra = 500 $&$2.4422\se{-1}$&$1.9907\se{-1}$&$2.0160\se{-1}$&$\mathbf{1.6571\se{-1}}$&$2.1718\se{-1}$&$1.7794\se{-1}$\\
$\Rey = 5, \Gra = 1000$&$8.4188\se{-1}$&$5.7008\se{-1}$&$6.1843\se{-1}$&$\mathbf{4.3465\se{-1}}$&$6.8417\se{-1}$&$4.7627\se{-1}$\\
$\Rey = 5, \Gra = 2000$&$1.8147       $&$9.8277\se{-1}$&$1.0915       $&$\mathbf{6.3786\se{-1}}$&$1.2169       $&$7.0854\se{-1}$
\end{tabular}
\caption{\small Comparison of $\mathrm{res}_{\Gamma_\oo}^{N}$ for solutions on the truncated domain $\Omega$ with the regarded open boundary conditions.}\label{fig:table_profile_difference_tables}
\end{subtable}
\caption{\small Comparison of $\mathrm{res}_{\Omega}^{N}$ and $\mathrm{res}_{\Gamma_\oo}^{N}$ for solutions on the truncated domain $\Omega$ with the regarded open boundary conditions.}
\end{table}
\end{landscape}

\section*{Acknowledgements}
C. N. Rautenberg has been supported via the framework of {\sc Matheon} by the Einstein Foundation Berlin within the ECMath project SE19 ``Optimal Network Sensor Placement for Energy Efficiency '' and acknowledges the support of the DFG through the DFG-SPP 1962: Priority Programme ``Non-smooth and Complementarity-based Distributed Parameter Systems: Simulation and Hierarchical Optimization'' phase 1 within Project 11 ``Optimal Control of Elliptic and Parabolic Quasi-Variational Inequalities'', and under Germany's Excellence Strategy - The Berlin Mathematics Research Center MATH+ (EXC-2046/1, project ID: 390685689) within project AA4-3.

\appendix
\section{Regularity of solutions to the evolutionary Stokes problem with mixed boundary conditions}

Although the following result may be obtained from classical ones, the
specific structure of the constants obtained within the bounds is hard
to be inferred from known theorems. In order to keep the paper self-contained,
the proof is given.

In what follows we assume that $\Omega $ is a subset of
$\mathbb{R}^{2}$ that satisfies condition $(\mathrm{A2})$ on page
\pageref{cond:A2}, and the boundary decomposition
$\{\Gamma _{\mathrm{i}},\Gamma _{\mathrm{w}},\Gamma _{\mathrm{o}}\}$ satisfies
the regularity and geometrical conditions assumed in  {Theorem~\ref{Evolutionary-Theorem1}}.
\begin{theorem}%
\label{Ap}
Let ${\mathbf{h}}\in L^{2}(0,T;L^{2}(\Omega )^{2})$ and
${\mathbf{w}}_{\mathrm{o}}\in V_{1}$. Then, the abstract evolutionary Stokes
problem
\begin{equation}
\label{Stokes}
\begin{split} &\langle \partial _{t}{\mathbf{w}}(t),{\mathbf{y}}\rangle +
\frac{1}{\mathrm{Re}}(\nabla {\mathbf{w}}(t),\nabla {\mathbf{y}})_{2}=({\mathbf{h}}(t),{
\mathbf{y}})_{2}\qquad \text{for all ${\mathbf{y}}\in V_{1}$ and a.e. }t\in (0,T),
\\
&{\mathbf{w}}(0)={\mathbf{w}}_{\mathrm{o}}
\end{split}
\end{equation}
has a unique solution ${\mathbf{w}}$, which belongs to
\begin{equation*}
\tilde{W}_{1}(0,T)=\{{\mathbf{w}}\in L^{2}(0,T;W^{2,2}(\Omega )^{2})\cap L^{\infty }(0,T;V_{1}), \partial _{t}{\mathbf{w}}\in L^{2}(0,T;L^{2}(\Omega )^{2})
\},
\end{equation*}
and satisfies the estimate
\begin{equation}
\label{estimate}
\begin{split} \|{\mathbf{w}}\|_{\tilde{W}_{1}(0,T)}:=\,&\|{\mathbf{w}}\|_{L^{2}(0,T;W^{2,2}(
\Omega )^{2})}+ \|{\mathbf{w}}\|_{L^{\infty }(0,T;V_{1})}+ \|\partial _{t}{
\mathbf{w}}\|_{L^{2}(0,T;L^{2}(\Omega )^{2})}
\\
\leq \,& c_{1}(\mathrm{Re})\|{\mathbf{h}}\|_{L^{2}(0,T;L^{2}(\Omega )^{2})}+c_{2}(
\mathrm{Re})\|{\mathbf{w}}_{\mathrm{o}}\|_{V_{1}},
\end{split}
\end{equation}
where
\begin{align*}
c_{1}(\mathrm{Re})&:=C(1+\mathrm{Re}^{1/2}+\mathrm{Re}),\qquad c_{2}(
\mathrm{Re}):=C(1+\mathrm{Re}^{-1/2}+\mathrm{Re}^{1/2}),
\end{align*}
 and $C$ is a positive constant that depends only on $\Omega $.
\end{theorem}

\begin{proof}
From Proposition III.2.3 of \cite{Sh1997}, we know that there exists a
unique ${\mathbf{w}}\in W_{1}(0,T)$ that solves problem  {\eqref{Stokes}}. The
increased regularity of ${\mathbf{w}}$, i.e., that
${\mathbf{w}}\in \tilde{W}_{1}(0,T)$, can be obtained by analogous arguments
to those leading to Theorem 5 of section 7.1.3 of \cite{Ev2010} in combination
with the regularity result for steady Stokes problems given in Theorem
A.1 of \cite{BeKu2016}. We describe them below.

Let $\{{\mathbf{y}}_{1},{\mathbf{y}}_{2},\ldots \}$ be an orthogonal basis of
$V_{1}$ that is also an orthonormal basis of $H_{1}$, and let
$m\in \mathbb{N}$. The same techniques for proving the existence and uniqueness
of Galerkin approximations for parabolic problems (see, e.g., Theorem 1
of section 7.1.2 of \cite{Ev2010}) yield the existence and uniqueness of
an element ${\mathbf{w}}_{m}$ of the form
\begin{equation}
\label{GalerkinApprox}
{\mathbf{w}}_{m}(t):=\displaystyle \sum _{k=1}^{m}d_{m}^{k}(t){\mathbf{y}}_{k}
\qquad \text{for a.e. }t\in (0,T),
\end{equation}
that satisfies
\begin{align}
\label{Eq}
&(\partial _{t}{\mathbf{w}}_{m}(t),{\mathbf{y}}_{k})_{2}+\frac{1}{\mathrm{Re}}(
\nabla {\mathbf{w}}_{m}(t),\nabla {\mathbf{y}}_{k})_{2}=({\mathbf{h}}(t),{\mathbf{y}}_{k})_{2},
\\
\label{IC}
&d_{m}^{k}(0)=({\mathbf{w}}_{\mathrm{o}},{\mathbf{y}}_{k})_{2},
\end{align}
for all ${\mathbf{y}}\in V_{1}$ and a.e. $t\in (0,T)$, where
$k=1,\ldots ,m$.

Let $t\in (0,T)$. We choose ${\mathbf{y}}=\partial _{t}{\mathbf{w}}_{m}(t)$ in  {\eqref{Eq}} to obtain that
\begin{equation*}
\|\partial _{t}{\mathbf{w}}_{m}(t)\|_{L^{2}(\Omega )^{2}}^{2}+
\frac{1}{\mathrm{Re}}(\nabla {\mathbf{w}}_{m}(t),\nabla \partial _{t}{
\mathbf{w}}_{m}(t))_{2}=({\mathbf{h}}(t),\partial _{t}{\mathbf{w}}_{m}(t))_{2}.
\end{equation*}
Then, taking into account that
\begin{equation*}
(\nabla {\mathbf{w}}_{m}(t),\nabla \partial _{t}{\mathbf{w}}_{m}(t))_{2}=
\frac{1}{2}\partial _{t}\|{\mathbf{w}}_{m}(t)\|_{V_{1}}^{2}
\end{equation*}
and that by Young's inequality we have
\begin{equation*}
|({\mathbf{h}}(t),\partial _{t}{\mathbf{w}}_{m}(t))_{2}|\leq
\frac{1}{2\varepsilon }\|\partial _{t}{\mathbf{w}}_{m}(t)\|_{L^{2}(\Omega )}^{2}+
\frac{\varepsilon }{2}\|{\mathbf{h}}(t)\|_{L^{2}(\Omega )^{2}}^{2},
\end{equation*}
for any $\varepsilon >0$, we obtain that
\begin{equation*}
\|\partial _{t}{\mathbf{w}}_{m}(t)\|^{2}_{L^{2}(\Omega )^{2}}+
\frac{1}{2\,\mathrm{Re}}\partial _{t}\|{\mathbf{w}}_{m}(t)\|_{V_{1}}^{2}
\leq \frac{1}{2\varepsilon }\|\partial _{t}{\mathbf{w}}_{m}(t)\|_{L^{2}(
\Omega )}^{2}+\frac{\varepsilon }{2}\|{\mathbf{h}}(t)\|_{L^{2}(\Omega )^{2}}^{2}.
\end{equation*}
Selecting $\varepsilon =1$, we find that
\begin{equation}
\label{EqNorms}
\|\partial _{t}{\mathbf{w}}_{m}(t)\|^{2}_{L^{2}(\Omega )^{2}}+
\frac{1}{\mathrm{Re}}\partial _{t}\|{\mathbf{w}}_{m}(t)\|_{V_{1}}^{2}
\leq \|{\mathbf{h}}(t)\|_{L^{2}(\Omega )^{2}}^{2}.
\end{equation}

Integration of  {\eqref{EqNorms}} from $0$ to $t$ yields
\begin{equation*}
\int _{0}^{t}\|\partial _{s}{\mathbf{w}}_{m}(s)\|^{2}_{L^{2}(\Omega )^{2}}
\,\mathrm{d}s+\frac{1}{\mathrm{Re}}\|{\mathbf{w}}_{m}(s)\|_{V_{1}}^{2}
\leq \int _{0}^{t}\|{\mathbf{h}}(s)\|_{L^{2}(\Omega )^{2}}^{2}\,
\mathrm{d}s+\frac{1}{\mathrm{Re}}\|{\mathbf{w}}_{m}(0)\|_{V_{1}}^{2}.
\end{equation*}
Thus, taking into account that
$\|{\mathbf{w}}_{m}(0)\|_{V_{1}}\leq \|{\mathbf{w}}_{\mathrm{o}}\|_{V_{1}}$ by  {\eqref{IC}}, we have
\begin{equation*}
\int _{0}^{t}\|\partial _{s}{\mathbf{w}}_{m}(s)\|^{2}_{L^{2}(\Omega )^{2}}
\,\mathrm{d}s+\frac{1}{\mathrm{Re}}\|{\mathbf{w}}_{m}(s)\|_{V_{1}}^{2}
\leq \int _{0}^{t}\|{\mathbf{h}}(s)\|_{L^{2}(\Omega )^{2}}^{2}\,
\mathrm{d}s+\frac{1}{\mathrm{Re}}\|{\mathbf{w}}_{\mathrm{o}}\|_{V_{1}}^{2},
\end{equation*}
and since $t\in (0,T)$ was arbitrary, we observe
\begin{equation}
\label{EqLimit}
\begin{split} &\int _{0}^{T}\|\partial _{t}{\mathbf{w}}_{m}(t)\|^{2}_{L^{2}(
\Omega )^{2}}\,\mathrm{d}t+\frac{1}{\mathrm{Re}}\sup \{\|{\mathbf{w}}_{m}(t)
\|_{V_{1}}^{2}:t\in (0,T)\}
\\
&\quad \leq \|{\mathbf{h}}\|_{L^{2}(0,T;L^{2}(\Omega )^{2}}^{2}+
\frac{1}{\mathrm{Re}}\|{\mathbf{w}}_{\mathrm{o}}\|_{V_{1}}^{2}.
\end{split}
\end{equation}

Passing to the limit as $m\to \infty $ in  {\eqref{EqLimit}} we obtain that
$\partial _{t}{\mathbf{w}}\in L^{2}(0,T;L^{2}(\Omega )^{2})$,
${\mathbf{w}}(t)\in L^{\infty }(0,T;V_{1})$, and that the following estimates
hold true:
\begin{align}
\label{Estimate-Derivative}
\|\partial _{t}{\mathbf{w}}\|_{L^{2}(0,T;L^{2}(\Omega )^{2})}^{2}\leq \,&
\|{\mathbf{h}}\|^{2}_{L^{2}(0,T;L^{2}(\Omega )^{2})}+
\frac{1}{\mathrm{Re}}\|{\mathbf{w}}_{\mathrm{o}}\|_{V_{1}}^{2},
\\
\label{Estimate-Linfty}
\|{\mathbf{w}}\|^{2}_{L^{\infty }(0,T;V_{1})}\leq \,& \mathrm{Re}\,\|{\mathbf{h}}
\|^{2}_{L^{2}(0,T;L^{2}(\Omega )^{2})}+\|{\mathbf{w}}_{\mathrm{o}}\|^{2}_{V_{1}}.
\end{align}

Let $t\in (0,T)$. Notice that ${\mathbf{w}}$ satisfies
\begin{equation*}
\frac{1}{\mathrm{Re}}(\nabla {\mathbf{w}}(t),\nabla {\mathbf{y}})_{2}=(
\tilde{{\mathbf{h}}}(t),{\mathbf{y}})_{2},
\end{equation*}
for all ${\mathbf{y}}\in V_{1}$, where
$\tilde{{\mathbf{h}}}(t):={\mathbf{h}}(t)-\partial _{t}{\mathbf{w}}(t)\in L^{2}(
\Omega )^{2}$. Therefore, it follows from Theorem A.1 of
\cite{BeKu2016} and assumptions on the boundary decomposition
$\{\Gamma _{\mathrm{i}},\Gamma _{\mathrm{w}},\Gamma _{\mathrm{o}}\}$ that
${\mathbf{w}}(t)\in W^{2,2}(\Omega )^{2}$, satisfying
\begin{equation*}
\begin{split} \|{\mathbf{w}}(t)\|_{W^{2,2}(\Omega )^{2}}\leq C\mathrm{Re}
\,\|\tilde{{\mathbf{h}}}(t)\|_{L^{2}(\Omega )},
\end{split}
\end{equation*}
 and thus,
\begin{equation*}
\begin{split} {\|{\mathbf{w}}(t)\|_{W^{2,2}(\Omega )^{2}}\leq C\mathrm{Re}
\,(\|{\mathbf{h}}(t)\|_{L^{2}(\Omega )^{2}}+\|\partial _{t}{\mathbf{w}}(t)\|_{L^{2}(
\Omega )^{2}})}.
\end{split}
\end{equation*}
Integrating the square of the above inequality with respect to $t$ from
$0$ to $T$, and using  {\eqref{Estimate-Derivative}}, we obtain that
\begin{equation*}
\begin{split} {\int _{0}^{T}\|{\mathbf{w}}(t)\|_{W^{2,2}(\Omega )^{d}}^{2}
\,\mathrm{d}t\leq \, C\mathrm{Re}^{2}\,\|{\mathbf{h}}\|^{2}_{L^{2}(0,T;L^{2}(
\Omega )^{2})}+ C\mathrm{Re}\,\|{\mathbf{w}}_{\mathrm{o}}\|_{V_{1}}^{2}},
\end{split}
\end{equation*}
that is, ${\mathbf{w}}\in L^{2}(0,T;W^{2,2}(\Omega )^{2})$ and
\begin{equation}
\label{EstimateW22}
{\|{\mathbf{w}}\|_{L^{2}(0,T;W^{2,2}(\Omega )^{2})}^{2}\leq C\mathrm{Re}^{2}
\,\|{\mathbf{h}}\|^{2}_{L^{2}(0,T;L^{2}(\Omega )^{2})}+ C\mathrm{Re}\,\|{
\mathbf{w}}_{\mathrm{o}}\|_{V_{1}}^{2}}.
\end{equation}

From  {\eqref{Estimate-Derivative}},  {\eqref{Estimate-Linfty}}, and  {\eqref{EstimateW22}} we find that ${\mathbf{w}}\in \tilde{W}_{1}(0,T)$, and in
addition, we have
\begin{equation*}
\begin{split} \|{\mathbf{w}}\|^{2}_{\tilde{W}_{1}(0,T)}\leq \,&C(1+
\mathrm{Re}+\mathrm{Re}^{2})\|{\mathbf{h}}\|^{2}_{L^{2}(0,T;L^{2}(
\Omega )^{2})}
\\
&+ C(1+\mathrm{Re}^{-1}+\mathrm{Re})\|{\mathbf{w}}_{\mathrm{o}}\|^{2}_{V_{1}},
\end{split}
\end{equation*}
which yields the desired estimate  {\eqref{estimate}}.
\end{proof}
\bibliographystyle{plain} 
\bibliography{References}
\end{document}